\documentclass[12pt,reqno]{amsart}
\pdfoutput=1
\usepackage{lipsum,mathtools,graphicx}
\usepackage{xspace,tikz-cd}
\usepackage[a4paper,margin=2.5cm]{geometry}
\usepackage[smalltableaux]{ytableau}
\usepackage[backend=bibtex, style=alphabetic, doi=false,isbn=false]{biblatex}
\newtheorem{thm}{Theorem}
\newtheorem{lemma}{Lemma}
\newtheorem{corollary}{Corollary}
\newtheorem{proposition}{Proposition}

\theoremstyle{definition}
\newtheorem{example}{Example}

\newtheorem{definition}{Definition}

\newcommand{\mfp}{{\scriptstyle{\mathfrak{P}}}}
\newcommand{\nminus}{\mathfrak{n}^-}
\newcommand{\lij}[1][ij]{\mathbb{L}_{#1}}
\newcommand{\chil}{\chi_{\Lambda_0}}
\newcommand{\xtil}{\overline{x}}
\newcommand{\qw}{$q$-Whittaker\xspace}

\newcommand{\wl}[1][q]{W_\lambda(X_n; #1)}
\newcommand{\wloc}[1][\lambda]{W_{\mathrm{loc}}(#1)}
\newcommand{\wlhat}[1][\lambda]{W^\dag_{#1}(X_n; q)}

\newcommand{\htl}{\widetilde{H}_{\lambda}(X_n; q,t)}
\newcommand{\naturals}{\mathbb{N}}
\newcommand{\scf}{\mathcal{F}(\lambda)}
\newcommand{\csf}[1][\lambda]{\mathrm{CSF}(#1)}

\newcommand{\GT}{\mathrm{GT}(\lambda)}
\newcommand{\qbinom}[2]{{#1 \brack #2}_q}
\newcommand{\pop}[1][\lambda]{\mathrm{POP}(#1)}

\newcommand{\barsig}{\overline{\sigma}}
\newcommand{\bartau}{\overline{\tau}}
\newcommand{\fdag}{F^\dag}
\newcommand{\ftilde}{\widetilde{F}}
\newcommand{\ctilde}[1][c]{\widetilde{#1}}
\newcommand{\spi}[1][i]{\mathbb{S}_{#1}}

\newcommand{\rootthree}{1.7320508}

\newcommand{\tildesig}{\tilde{\sigma}}
\newcommand{\tildetau}{\tilde{\tau}}

\newcommand{\ihat}{\widehat{I}\,(x,y,z)}
\newcommand{\up}[1][]{{#1}^\uparrow}
\newcommand{\down}[1][]{{#1}^\downarrow}
\newcommand{\afsl}[1][n]{\widehat{\mathfrak{sl}_{#1}}}
\newcommand{\curr}[1][n]{\mathfrak{sl}_{#1}[t]}
\newcommand{\sln}[1][n]{\mathfrak{sl}_{#1}}
\newcommand{\SSYT}[1][\lambda]{\mathrm{SSYT}(#1)}
\newcommand{\Res}[2]{\mathrm{Res}^{#1}_{#2}}
\newcommand{\complex}{\mathbb{C}}
\newcommand{\elp}{\mathbb{E}(\lambda^\prime)}
\DeclareMathOperator{\supp}{supp}
\DeclareMathOperator{\len}{len}
\DeclareMathOperator{\inv}{inv}
\DeclareMathOperator{\refinv}{refinv}
\DeclareMathOperator{\quinv}{quinv}
\DeclareMathOperator{\maj}{maj}
\DeclareMathOperator{\dg}{dg}
\DeclareMathOperator{\dgh}{\widehat{\dg}}
\DeclareMathOperator{\NE}{NE}
\DeclareMathOperator{\SE}{SE}
\DeclareMathOperator{\br}{br}
\DeclareMathOperator{\pr}{pr}
\DeclareMathOperator{\bcomp}{boxcomp}
\DeclareMathOperator{\Comp}{Comp}
\DeclareMathOperator{\area}{area}
\DeclareMathOperator{\rsort}{rsort}

\DeclareMathOperator{\upop}{up}
\DeclareMathOperator{\wt}{wt}
\DeclareMathOperator{\zcount}{zcount}
\DeclareMathOperator{\zcb}{\overline{\zcount}}
\DeclareMathOperator{\coarm}{coarm}
\DeclareMathOperator{\arm}{arm}
\DeclareMathOperator{\Inv}{Inv}
\DeclareMathOperator{\Des}{Des}
\DeclareMathOperator{\cells}{cells}
\DeclareMathOperator{\splice}{splice}
\DeclareMathOperator{\dsplice}{dsplice}
\DeclareMathOperator{\cand}{cand}

\DeclareMathOperator{\CL}{CL}
\DeclareMathOperator{\uea}{\mathfrak{U}}
\definecolor{cadmiumgreen}{rgb}{0.0, 0.42, 0.24}
\definecolor{chocolate}{rgb}{0.48, 0.25, 0.0}
\definecolor{darkpastelgreen}{rgb}{0.01, 0.75, 0.24}
\definecolor{cadmiumred}{rgb}{0.89, 0.0, 0.13}
\definecolor{orangered}{rgb}{1.0, 0.27, 0.0}
\definecolor{light-gray}{gray}{0.9} 
\colorlet{Mycolor1}{green!80!magenta}
\colorlet{Mycolor6}{lime!80!orange}
\colorlet{Mycolor3}{red!90!yellow}
\colorlet{Mycolor7}{yellow!30!blue}
\colorlet{Mycolor5}{orange!80!red}
\colorlet{Mycolor4}{blue!40!pink}
\colorlet{Mycolor8}{yellow!80!magenta}
\definecolor{Mycolor2}{rgb}{0.5, 0.42, 0.24}

\begin{document}
\title[]{$q$-Whittaker polynomials: \\bases, branching and direct limits}
\author[Bhattacharya]{Aritra Bhattacharya}
\address{Beijing International Center for Mathematical Research, Peking University, Beijing 100871, China}
\email{matharitra@gmail.com}
\thanks{The second and third authors acknowledge partial support under a DAE Apex Grant to the Institute of Mathematical Sciences.}
\author[Ratheesh]{T V Ratheesh}
\address{The Institute of Mathematical Sciences, A CI of Homi Bhabha National Institute, Chennai 600113, India}
\email{ratheeshtv@imsc.res.in}
\author[Viswanath]{Sankaran Viswanath}
\address{The Institute of Mathematical Sciences, A CI of Homi Bhabha National Institute, Chennai 600113, India}
\email{svis@imsc.res.in}
\begin{abstract}
We study \qw polynomials and their monomial expansions given by the fermionic formula, the {\em inv} statistic of Haglund-Haiman-Loehr and the {\em quinv} statistic of Ayyer-Mandelshtam-Martin. The combinatorial models underlying these expansions are {\em partition overlaid patterns} and {\em column strict fillings}. The former model is closely tied to representations of the affine Lie algebra $\afsl$ and admits projections, branching maps and direct limits that mirror these structures in the Chari-Loktev basis of local Weyl modules. We formulate novel versions of these notions  in the column strict fillings model and establish their main properties. We construct weight-preserving bijections between the models which are compatible with projection, branching and direct limits. We also establish connections to the {\em coloured lattice paths} formalism for \qw polynomials due to Wheeler and collaborators.
\end{abstract}
\maketitle

\section{Introduction}
The principal results of this paper were announced at FPSAC 2024 \cite{brv-fpsac}.
For $n \geq 1$, let $X_n$ denote the tuple of indeterminates $x_1, x_2, \cdots, x_n$. The \qw polynomial $\wl$ is the specialization of the Macdonald symmetric polynomial $P_\lambda(X_n; q,t)$ at $t=0$. It interpolates between the Schur polynomial $s_\lambda$ and the elementary symmetric polynomial $e_{\lambda'}$, reducing to these at $q=0$ and $q=1$ respectively. The \qw polynomial is closely related to the modified Hall-Littlewood polynomial, being obtained from the latter via the classical involution $\omega$ on symmetric polynomials.

In addition to the important role it plays in the theory of symmetric functions, it occurs in diverse other contexts - as the Frobenius characteristic of a graded component of the Garsia-Haiman module \cite{bergeron2020survey}, as the counting function of flags of subspaces over finite fields compatible with a linear endomorphism \cite{hugh-thomas,samrith2,samrith1} and as the partition function of certain ensembles of coloured lattice paths \cite{borodin-wheeler, GarbaliWheeler}, to name a few.

The viewpoint of this paper stems from the connection of \qw polynomials to the representation theory of the affine Lie algebra $\widehat{\mathfrak{sl}_n}$. Let $\lambda$ denote a dominant integral weight of  $\sln$, which we will identify with a partition having at most $n-1$ nonzero parts. In the $\afsl$ setting, the \qw polynomials $\wl$ coincide with the graded character of $\mathfrak{sl}_n$-invariant Demazure modules of the level 1 vacuum module (or basic representation) $L(\Lambda_0)$ of  $\widehat{\mathfrak{sl}_n}$. Such Demazure modules are also indexed by dominant integral weights $\lambda$ of $\sln$, and may equivalently be viewed as graded modules for the maximal parabolic subalgebra  $\mathfrak{sl}_n[t]$ (the {\em current algebra}) of $\afsl$. They are also called {\em local Weyl modules} in this context \cite{ChariLoktev-original, fourier-littelmann}. The graded character of the local Weyl module $\wloc$ is $\wl$ \cite[Corollary 1.5.2]{ChariLoktev-original}, \cite{ChariIon-BGG,ion-demazure,sanderson}.

Setting $q=1$ in $\wl$ is tantamount to ignoring the grading and restricting $\wloc$ to a module over $\sln \subseteq \curr$. In this case, we have $\Res{\curr}{\sln} \wloc \cong \elp$ where
    \begin{equation}\label{eq:elam}
 \elp :=\Lambda^{\lambda_1^\prime} V \otimes \Lambda^{\lambda_2^\prime} V \otimes \cdots \otimes \Lambda^{\lambda_k^\prime} V
    \end{equation}
    where $V = \complex^n$ is the standard representation of $\sln$ \cite{ChariLoktev-original,fourier-littelmann} and $\lambda'_i\, (1 \leq i \leq k)$ denote the parts of the conjugate partition $\lambda'$. The character of 
    $\elp$ is the elementary symmetric polynomial
    \begin{equation}\label{eq:elemsym}
          \wl[1] = e_{\lambda^\prime}(X_n) = \sum_{F \in \csf} x^F, 
    \end{equation}
    where $\csf$ denotes the natural parameterizing set of $\elp$ consisting of column-strict fillings of shape $\lambda$ (see section~\ref{sec:modmac}). 
The following identity due to Haglund, Haiman and Loehr  \cite{HHL-I} for the \qw polynomial\footnote{Specializing their more general identity for the modified Macdonald polynomial.} is a $q$-analog of \eqref{eq:elemsym}: 
\begin{equation}\label{eq:wlinv}
  \wl = \sum_{F \in \csf} x^F q^{\inv(F)} 
\end{equation}
where the statistic {\em inv} is defined in equation~\eqref{eq:inveq} below.

There is also a ``$q$-analog" construction at the level of modules. The local Weyl module $\wloc$ may be viewed as a graded deformation of $\elp$ via the fusion product construction, which introduces a filtration, and thereby an associated grading, on the ungraded module $\elp$. We have \cite[Corollary 1.5.1]{ChariLoktev-original}, \cite[Corollary A]{fourier-littelmann}:
    \begin{equation}\label{eq:fusion}
    \wloc =  \Lambda^{\lambda_1^\prime} V * \Lambda^{\lambda_2^\prime} V * \cdots * \Lambda^{\lambda_k^\prime} V
    \end{equation}
    where $*$ denotes the fusion product of these modules.  The fusion module construction \eqref{eq:fusion} may be viewed as a $q$-analog of the tensor product  \eqref{eq:elam} in so far as it converts the underlying $\sln$-module $\elp$ into a graded module (over $\curr$), with $q$ keeping track of the grades. 

    Another combinatorial model in the same spirit as \eqref{eq:wlinv} is afforded by Ayyer, Mandelshtam and Martin's {\em quinv} statistic\footnote{More generally, this is defined on all fillings of the Young diagram of $\lambda$.}  on $\csf$. In close parallel to \eqref{eq:wlinv}, we have \cite{AMM}:
\begin{equation}\label{eq:wlquinv}
  \wl = \sum_{F \in \csf} x^F q^{\quinv(F)} 
\end{equation}

    Equations~\eqref{eq:wlinv}-\eqref{eq:wlquinv} suggest that there should be natural bases of $\wloc$ indexed by $\csf$, comprising homogeneous elements whose grades are encoded by the inv or quinv statistics. One of the aims of this paper is to construct such bases; in Proposition~\ref{prop:clbasis-csf-Et}, we use the work of Chari-Loktev (described below) to obtain monomial bases of $\wloc$ specified natively in terms of $\csf$, with precisely the above properties. 

Besides its realization as a fusion product, $\wloc$ also admits a presentation, with a single generator and a concise list of relations \cite[Definition 1.2.1]{ChariLoktev-original}. This vantage point allowed Chari and Loktev to construct a remarkable PBW type basis of $\wloc$ \cite[Theorem 2.1.3]{ChariLoktev-original}. The Chari-Loktev (CL) basis has branching properties for the restriction $\curr[n-1] \subseteq \curr$ that are akin to those enjoyed by the classical Gelfand-Tsetlin basis (of irreducible representations of $\sln$) for the restriction $\sln[n-1] \subseteq \sln$. 

The Chari-Loktev indexing set was later reformulated in \cite{RRV-CLpop} as pairs $(T,\Lambda)$ comprising a Gelfand-Tsetlin pattern $T$ and a tuple $\Lambda$ of partitions whose Young diagrams ``fit" into rectangular boxes of appropriate sizes at each node of $T$. This {\em partition overlaid pattern} (POP) model yields a visual interpretation of yet another monomial expansion of the \qw polynomial \cite[Proposition 2.1.4]{ChariLoktev-original},\cite[(7.13$'$)]{MacMainBook} - the so-called {\em fermionic formula} (see equation~\eqref{eq:fermpop} and Figure~\ref{fig:gt-ssyt} below). The key feature of the POP model is the existence of a natural combinatorial branching map (see equation~\eqref{eq:brdef}):
\[ \pop \to \bigsqcup_{\mu \prec \lambda} \pop[\mu] \]
that encodes the branching property of the CL basis. 

Finally, in addition to its characterizations in terms of generators-relations  and as a fusion product, $\wloc$ can also be realized as a Demazure module, as was alluded to previously.  This perspective uncovers another facet of these modules - that  there is a chain of inclusions 
\[ \wloc \hookrightarrow \wloc[\lambda+\theta] \hookrightarrow \cdots \hookrightarrow \wloc[\lambda+k\theta] \hookrightarrow \cdots \]
where $\theta$ denotes the highest root of $\sln$ \cite{fourier-littelmann,RRV-CLpop}. This follows from the fact that the affine Weyl group element corresponding to the extremal generator of the Demazure module $\wloc[\lambda+k\theta]$ increases in Bruhat order as $k$ increases. The direct limit of this chain is the entire module $L(\Lambda_0)$ of $\afsl$. This yields a (subtle) notion of combinatorial direct limits for partition overlaid patterns, formulated in \cite[\S 5]{RRV-CLpop}, with 
\[\varinjlim_k \pop[\lambda+k\theta]\] giving a parameterizing set for the character of $L(\Lambda_0)$. 

In contrast to the POP model, the CSF model considered earlier has no fruitful notion thus far of branching or direct limits. It remains, in this sense, a model fundamentally rooted in the finite (i.e., in $\sln$), albeit one that is augmented with $q$-grading data (via inv or quinv). The primary goal of this paper, and arguably its core novelty, lies in the precise formulation of branching and direct limits for column strict fillings.
Our {\em splice} and {\em dsplice} operations (section~\ref{sec:probrancsf}) provide the branching formalism for column strict fillings, while the sequence of injective maps in equation~\eqref{eq:csfdl} enables passage to the direct limit. The splice operation has many novel features of independent interest.

Our main theorem (Theorem~\ref{thm:mainthm}) constructs two remarkable bijections from $\csf$ to $\pop$, one each for inv and quinv, which respect all structures - the $x$-weights, $q$-weights and the branching and projection maps of the CSF and POP models.
Along the way, we will have occasion to resurrect a close cousin of {\em inv} and {\em quinv}, first formulated by Zelevinsky\footnote{in the appendix he wrote to the Russian translation of Macdonald's monograph on Symmetric Functions. This statistic is denoted `ZEL' by Kirillov in \cite{Kirillov_newformula}.} in 1984. This statistic, which we term {\em refinv}, plays a central role in our proofs.

As a corollary to our main theorem, we establish that the sum $(inv+quinv)$ is constant on fibers of the natural projection map. More surprisingly, when inv and quinv are viewed as statistics on partition overlaid patterns  via our bijections, they turn out to be mutually related by the classical {\em box-complementation} involution of the POP model. This partially answers a question posed by Ayyer-Mandelshtam-Martin \cite{AMM} on constructing a bijection between the inv and quinv models that leaves the entries of each row of the filling invariant (Corollary \ref{cor:bijamm}). 

As mentioned earlier, Theorem~\ref{thm:mainthm} allows us to repurpose the Chari-Loktev basis into an elegant basis that is native to the CSF formalism (Proposition~\ref{prop:clbasis-csf-Et}). 
We also use our bijections in conjunction with the results of \cite{RRV-CLpop} to construct direct limits and to thereby obtain a new character formula for $L(\Lambda_0)$ in terms of column strict fillings (Proposition ~\ref{prop:charcsf}, Corollary~\ref{cor:kthetal}).

Finally, we make connections to the {\em coloured lattice paths} model of \qw polynomials\footnote{More generally, for the modified Macdonald polynomial.} stemming from work of Garbali-Wheeler \cite{GarbaliWheeler}, Borodin-Wheeler \cite{borodin-wheeler} and Wheeler \cite{wheeler-lectures}. We show that a refined counting of intersections and non-intersections of paths furnishes a succinct visual representation of our main theorem - a  ``proof without words" of sorts - that provides a simultaneous readout of both bijections. 

The results of this paper have parallels for the modified Hall-Littlewood polynomials as well. These, together with their quasisymmetric generalizations, will be the subject of future work.

The paper is organised as follows. After the preliminaries in sections \ref{sec:modmac}-\ref{sec:pop-qbinom}, we formulate our projection (rowsort) and branching (dsplice) operations for column strict fillings in section \ref{sec:probrancsf}. Section~\ref{sec:mainresults} contains the statements of our main results. Sections~\ref{sec:qtrip} and \ref{sec:invreftrip} define the two bijections and establish their basic properties, while section~\ref{sec:compf} completes the proof of Theorem~\ref{thm:mainthm} and Proposition~\ref{prop:clbasis-csf-Et}. Section~\ref{sec:dirlim} constructs the direct limit in the CSF model and derives a new character formula for $L(\Lambda_0)$. Finally, section~\ref{sec:lattpathwheeler} interprets our main theorems via the coloured lattice paths formalism.

\section{The modified Macdonald polynomial $\htl$}\label{sec:modmac}
Given a partition $\lambda=(\lambda_1 \geq \lambda_2 \geq \cdots)$, we will draw its Young diagram $\dg(\lambda)$ following the English convention, as a left-up justified array of boxes, with $\lambda_i$ boxes in the $i$th row from the top. The boxes are called the cells of $\dg(\lambda)$; each cell $c$ is assigned a co-ordinate $(c_1,c_2)$ where $c_1$ is the row index counting from the top and $c_2$ is the column index counting from left.
We let $|\lambda|:=\sum_i \lambda_i$.

\subsection{} Fix $n \geq 1$ and let $\scf$ denote the set of all maps (``fillings'') $F: \dg(\lambda) \to [n]$ where $[n]=\{1, 2, \cdots, n\}$. We will sometimes use $\dg(F)$ to denote $\dg(\lambda)$ for $F \in \scf$. If the values of $F$ strictly increase as we move down a column, we say $F$ is a {\em column strict filling} (CSF) and denote the set of such fillings by $\csf$. Let $X_n$ denote the tuple of indeterminates $x_1, x_2, \cdots, x_n$. The $x$-weight of a filling $F$ is the monomial $$x^F:=\displaystyle\prod_{c \in \dg(\lambda)} x_{F(c)}.$$

Given a cell $u=(i,j)$ of $\dg(\lambda)$ with $i>1$, we let $\up[u]$ denote the cell immediately above $u$, i.e., with coordinates 
$(i-1, j)$. 
A \textit{descent} of a filling $F \in \scf$ is a cell $u$ (not in the first row) such that $F(u) > F(\up[u])$.
 Define
\begin{equation}
\label{eq 1.0.5}
\Des(F) = \{u \in \dg(\lambda) : u \text{ is a descent of } F\}.
\end{equation}

\subsection{}
Given cells $u, v$ of $\dg(\lambda)$, we say that $u$ {\em Inv-attacks} $v$ if
\begin{enumerate}
\item they are in the same row, with $u$ strictly to the left of $v$, or
\item they are in consecutive rows, with $u$ occurring in the row below $v$, in a column that is strictly to the right of the column of $v$.
\end{enumerate}
An \textit{inversion} of $F \in \scf$ is a pair of cells $(u,v)$ in $\dg(\lambda)$ such that $u$ Inv-attacks $v$ and $F(u) > F(v)$. 
%when we read the cells of $\dg(\lambda)$ from left to right along each row, Taking the rows from bottom to top. 
Define
\begin{equation}
\label{eq1.0.6}
\Inv(F) = \{ (u, v): (u,v) \text { is an inversion of } F\}
\end{equation}
\begin{equation}
\label{eq:inveq}
\inv(F) = |\Inv(F)| - \sum_{u \in \Des(F)} \arm(u),
\end{equation} 
where ${\arm(u)}$ denotes the number of cells of $\dg(\lambda)$ strictly to the right of $u$ in the same row. For later use, we also define the coarm of a cell $u \in \dg(\lambda)$ - this is the number of cells of $\dg(\lambda)$ strictly to the left of $u$ in the same row.

\subsection{} %Do likewise for quinv.
Analogously, a cell $u \in \dg(\lambda)$ {\em Quinv-attacks} $v \in \dg(\lambda)$ if
\begin{enumerate}
\item $u$ is to the right of $v$ in  the same row, or 
\item they are in consecutive rows, with $u$ occurring in the row below $v$, in a column that is strictly to the left of the column of $v$.
\end{enumerate}

The attack relations are illustrated below.
\begin{align}\label{eq:attackrelspic}
	\ydiagram[*(gray)]{2,3+3}*[*(white) u]{0,2+1}*[*(white)]{7,6} \qquad \qquad \qquad & \qquad \qquad \ydiagram[*(gray)]{3+4,2}*[*(white) u]{0,2+1}*[*(white)]{7,6}
	\\
	\hbox{boxes Inv-attacked by $u$} \qquad \qquad & \qquad \hbox{ boxes Quinv-attacked by $u$}
	\nonumber
\end{align}

A \textit{queue-inversion} \cite{AMM} of $F \in \scf$ is a pair of cells $(u,v)$ in $\dg(\lambda)$ such that $u$ Quinv-attacks $v$ and  $F(u) > F(v)$. 
Define
\begin{equation}
%\label{eq1.0.6}
\mathrm{Quinv}(F) = \{ (u, v): (u,v) \text { is a queue-inversion of } F\}
\end{equation}
\begin{equation}
%\label{eqeq:inveq}
\quinv(F) = |\mathrm{Quinv}(F)| - \sum_{u \in \Des(F)} \arm(\up[u]).
\end{equation} 
We note that Ayyer-Mandelshtam-Martin first define quinv in terms of the so-called {\em quinv-triples} (see \S\ref{sec:qtrip} below); the above definition of $\quinv$ occurs as equation~(16) of \cite{AMM}.

\subsection{} %Define maj.

For a box $u$ in $\dg(\lambda)$ let $\mathrm{leg}(u)$ be the number of boxes in the same column strictly below $u$. For a filling $F \in \scf$, define \cite{HHL-I}: 
\begin{equation}\label{eq:majdef}
    \maj(F) = \sum_{u \in \Des(F)} (\mathrm{leg}(u)+1).
\end{equation} 
It follows that the largest value that $\maj$ can take is 
\[n(\lambda) := \sum_{j \geq 1} \binom{\lambda'_j}{2}\] 
where $\lambda'_j$ are the parts of the partition conjugate  to $\lambda$. 
The following lemma follows directly from the definition of $\maj$ in \eqref{eq:majdef}.
\begin{lemma}\label{lem:majminmax}
  Let $F \in \scf$. Then $\maj(F) = n(\lambda)$ if and only if $F \in \csf$. \qed
\end{lemma}

\subsection{} We recall that the modified Macdonald polynomial $\htl$ is a symmetric polynomial in the $x_i$ with $\naturals[q,t]$ coefficients \cite{HHL-I}. Following Haglund-Haiman-Loehr \cite{HHL-I} and Ayyer-Mandelshtam-Martin \cite{AMM}, it has the following monomial expansions in terms of the statistics {\em inv, quinv} and {\em maj} on $\scf$:
\begin{equation}\label{eq:invquinvmaj}
\htl = \sum_{F \in \scf} x^F q^{v(F)} t^{\maj(F)}
\end{equation}
where $v \in \{\inv, \quinv\}$.

\section{Monomial expansions of the \qw polynomial}
\subsection{inv and quinv expansions}
The \qw polynomial $\wl$ is the coefficient of the highest power of $t$ in the $t$-expansion below \cite[(3.1)]{bergeron2020survey}:
\begin{equation}\label{eq:berg1}
  \htl = \mathcal{H}_\lambda(X_n; q) \,t^0 + \cdots + \wl \, t^{n(\lambda)}.
  \end{equation}
  It also coincides with the specialization of the Macdonald polynomial $P_\lambda(X_n, q, t)$ at $t=0$ \cite{bergeron2020survey}.
Putting together \eqref{eq:invquinvmaj}, \eqref{eq:berg1} and Lemma~\ref{lem:majminmax}, we obtain for $v \in \{\inv,\quinv\}$:
\begin{equation}\label{eq:wlinvquinv}
  \wl = \sum_{F \in \csf} x^F q^{v(F)} 
\end{equation}
Since the $\wl$ is in fact symmetric in the $x$-variables,  \eqref{eq:wlinvquinv} can be viewed as an expansion of $\wl$ in terms of the monomial symmetric functions in $x_1, x_2, \cdots, x_n$.

\subsection{Gelfand-Tsetlin patterns}\label{sec:fermfor}

Let $n \geq 1$ and $\lambda = (\lambda_1 \geq \lambda_2 \geq \cdots \geq \lambda_n \geq 0)$ be a partition with at most $n$ nonzero parts. We now recall the definition of Gelfand-Tsetlin (GT) patterns. Consider a tuple $T = \left(T^j_i: \,1 \leq i \leq j \leq n\right)$  of integers, arranged in a triangular array as in Figure~\ref{fig:gt-ssyt}. 

We define the North-East and South-East differences of $T$ by:
\begin{equation}\label{eq:ne-se-def}
\NE_{ij}(T) = T^{j+1}_i - T^j_i \text{ ; }\SE_{ij}(T)= T^j_i - T^{j+1}_{i+1}
\end{equation}
for $1 \leq i \leq j < n$. For convenience, we let $T^j_{j+1}:=0$ for all $1 \leq j \leq n$ and use this in \eqref{eq:ne-se-def} to extend the definitions of $\NE_{ij}(T)$ and $\SE_{ij}(T)$ to the case $i=j+1$. 

We say that the tuple $T$ is a \textit{Gelfand-Tsetlin pattern} with bounding row (or shape) $\lambda$ if:
\begin{enumerate}
\item $\NE_{ij}(T) \geq 0$ and $\SE_{ij}(T)  \geq 0$ for $1 \leq i \leq j < n$, and
\item $T_i^n=\lambda_i$ for $1\leq i \leq n$.
\end{enumerate} 
Let $\GT$ denote the set of GT patterns of shape $\lambda$.
As is customary, we will interchangeably think of a GT pattern as a semistandard Young tableau (SSYT). In this perspective, $T^j_i$ is the number of cells in the $i^{th}$ row of the associated tableau which contain entries $\leq j$. 
Thus, $\NE_{ij}(T)$ is the number of cells in the $i^{th}$ row of the tableau which contain the entry $j+1$. Letting $\SSYT$ denote the set of semistandard Young tableaux of shape $\lambda$, we identify $\SSYT \cong \GT$ via the above.

For $T \in \GT$, let $x^T$ denote the $x$-weight of the corresponding tableau. We thus have:
\begin{equation}\label{eq:schurpoly}
    s_\lambda(X_n) = \sum_{T \in \GT} x^T
\end{equation}
where $s_\lambda(X_n)$ is the Schur polynomial. 

\subsection{Fermionic formula for $\wl$} The following ``fermionic" formula (i.e., one involving sums of products of $q$-binomials) for the \qw polynomial appears in \cite{HKKOTY, Kirillov_newformula} and follows readily by specializing a more general formula \cite[Chap VI, (7.13)$^\prime$]{MacMainBook} for the monomial expansion of the Macdonald $P$-function:
\begin{equation}\label{eq:ferm}
  \wl = \sum_{T \in \GT} x^T \prod_{1 \leq i \leq j <n}\qbinom{\NE_{ij}(T) + \SE_{ij}(T)}{\NE_{ij}(T)},
\end{equation}
where $\qbinom{\cdot}{\cdot}$ denotes the $q$-binomial coefficient. Following \cite{karpthomas}, we define 
\[\wt_q(T) := \displaystyle\prod_{1 \leq i \leq j <n}\displaystyle\qbinom{\NE_{ij}(T) + \SE_{ij}(T)}{\NE_{ij}(T)}.\]

\section{Partition overlaid patterns}\label{sec:pop-qbinom}
\subsection{} We recall that a partition $\gamma$ is said to {\em fit} into a rectangle of size $k \times \ell$ if $\gamma$ has at most $k$ nonzero parts, each of which is $\leq \ell$.
As is well-known, the $q$-binomial coefficient $\qbinom{k+\ell\,}{k}$ is the generating function of partitions that fit into a $k \times \ell$ rectangle:  
\begin{equation}\label{eq:qbindef}
\qbinom{k+\ell}{k} = \sum_\gamma q^{|\gamma|},
\end{equation}
where $\gamma = (\gamma_1 \geq \gamma_2\geq \cdots \geq\gamma_k \geq 0)$ with  $\ell \geq \gamma_1$. 
Let $A_{k\ell}$ denote the set of such partitions $\gamma$.
For later use, we also introduce the set $B_{k\ell}$ comprising the strictly decreasing $k$-tuples \[a=(a_1>a_2> \cdots > a_k \geq 0), \text{ with }k+\ell-1 \geq a_1.\] 
One sees readily that $A_{k\ell}$ and $B_{k\ell}$ are in bijection via:
\begin{equation}\label{eq:gamma-a}
  \gamma \to a \;\; \text{ where } \gamma_p  = a_p - (k -p) \text{ for all } 1 \leq p \leq k
  \end{equation}
As shown in \cite{RRV-CLpop}, the right-hand side of \eqref{eq:ferm} can be interpreted in terms of {\em partition overlaid patterns}. We recall the definition:

\begin{definition}\label{def:popdef}
  A partition overlaid pattern (POP) of shape $\lambda$ is a pair $(T, \Lambda)$ where $T \in \GT$ and $\Lambda=(\Lambda_{ij}: 1 \leq i \leq j <n)$ is a tuple of partitions such that each $\Lambda_{ij}$  fits into a rectangle of size $\NE_{ij}(T) \times \SE_{ij}(T)$. We let $|\Lambda| = \sum_{i,j} |\Lambda_{ij}|$.
\end{definition}
  For example, if $T$ is the GT pattern of Figure~\ref{fig:gt-ssyt}, we could take $\Lambda_{11} = (2,1,0),\;  \Lambda_{12} = (2),\;  \Lambda_{13} =(1,1),\;  \Lambda_{22} = (0,0,0) ,\; \Lambda_{23} = (1), \; \Lambda_{33} = (2,2)$. We imagine the $\Lambda_{ij}$ as being placed in a triangular array as in Figure~\ref{fig:gt-ssyt}.
We let $\pop$ denote the set of partition overlaid patterns of shape $\lambda$.
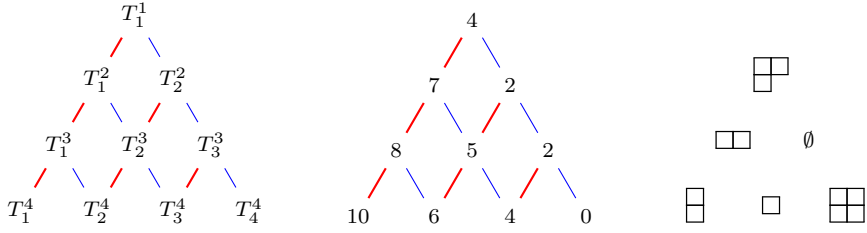
\begin{figure}
  \begin{center}
\begin{tikzpicture}[x={(1cm*0.5,-\rootthree cm*0.5)},y={(1cm*0.5,\rootthree cm*0.5)}]
  \foreach\i in{0,...,3}
\foreach\j in{\i,...,3}{
  \pgfmathtruncatemacro{\k}{4 - \j + \i};
  \pgfmathtruncatemacro{\l}{\i + 1};
  \draw(\i,\j)node(a\i\j){$\scriptstyle{T^{\k}_{\l}}$};
}
  \foreach\i/\ii in{0/-1,1/0,2/1,3/2}
 \foreach\j/\jj in{0/-1,1/0,2/1,3/2}{
  \ifnum\i<\j \draw[color=red,thick](a\i\jj)--(a\i\j); \fi
  \ifnum\i>\j\else\ifnum\i>0 \draw[color=blue](a\ii\j)--(a\i\j);\fi\fi
 }
\end{tikzpicture}\qquad
\begin{tikzpicture}[x={(1cm*0.5,-\rootthree cm*0.5)},y={(1cm*0.5,\rootthree cm*0.5)}]
  \draw(0,0)node(a00){$\scriptstyle{10}$};\draw(0,1)node(a01){$\scriptstyle{8}$};\draw(0,2)node(a02){$\scriptstyle{7}$};\draw(0,3)node(a03){$\scriptstyle{4}$};
  \draw(1,1)node(a11){$\scriptstyle{6}$};\draw(1,2)node(a12){$\scriptstyle{5}$};\draw(1,3)node(a13){$\scriptstyle{2}$};
  \draw(2,2)node(a22){$\scriptstyle{4}$};\draw(2,3)node(a23){$\scriptstyle{2}$};
  \draw(3,3)node(a33){$\scriptstyle{0}$};
  \foreach\i/\ii in{0/-1,1/0,2/1,3/2}
 \foreach\j/\jj in{0/-1,1/0,2/1,3/2}{
  \ifnum\i<\j \draw[color=red,thick](a\i\jj)--(a\i\j); \fi
  \ifnum\i>\j\else\ifnum\i>0 \draw[color=blue](a\ii\j)--(a\i\j);\fi\fi
}
\end{tikzpicture}
%% \qquad
%% \begin{tikzpicture}
%% \ytableausetup{mathmode, boxsize=0.3em}
%%   \draw(0,0) node {
%% \ytableaushort{
%%   1 1 1 1 2 2 2 3 4 4, 
%%   2 2 3 3 3 4,
%%   3 3 4 4}};
%% \end{tikzpicture}
\qquad
\begin{tikzpicture}[x={(1cm*0.5,-\rootthree cm*0.5)},y={(1cm*0.5,\rootthree cm*0.5)}]
    \ytableausetup{mathmode, boxsize=0.5em}
  \draw(0,0)node {\ydiagram{1,1}};\draw(0,1)node{\ydiagram{2}};\draw(0,2)node{\ydiagram{2,1}};
  \draw(1,1)node {\ydiagram{1}};\draw(1,2)node{$\scriptstyle{\emptyset}$};
  \draw(2,2)node {\ydiagram{2,2}};
\end{tikzpicture}
\caption{A GT pattern for $n=4$. The NE and SE differences are those along the red and blue lines. On the right is a partition overlay compatible with this GT pattern.}
\label{fig:gt-ssyt}
\end{center}
\end{figure}
%

%\noindent
It is now clear from \eqref{eq:ferm}, \eqref{eq:qbindef} and Definition~\ref{def:popdef} that for all $T \in \GT$,
\begin{equation}\label{eq:popfiber}
    \sum_{\Lambda:\; (T,\Lambda) \in \pop} q^{|\Lambda|} = \wt_q(T)
\end{equation} 
and hence 
\begin{equation}\label{eq:fermpop}
  \wl = \sum_{(T, \Lambda) \in \pop} x^T q^{|\Lambda|}.
  \end{equation}

\subsection{Projection and Branching for Partition overlaid patterns}
\begin{definition} 
Given $\lambda = (\lambda_1 \geq \lambda_2 \geq \cdots \geq \lambda_n \geq 0)$, we say that $\mu = (\mu_1, \mu_2, \cdots, \mu_{n-1})$ {\em interlaces} $\lambda$ (and write $\mu \prec \lambda$) if $\lambda_i \geq \mu_i \geq \lambda_{i+1}$ for $1 \leq i <n$. 
\end{definition}
Recall that the GT inequalities \eqref{eq:ne-se-def} imply that each row of a GT pattern interlaces the one below, i.e., if $T \in \GT$, then $(T^j_i: \,1 \leq i \leq j)$ interlaces $(T^{j+1}_i: \, 1 \leq i \leq j+1)$  for each $1 \leq j <n$.

The \qw polynomials have the following important properties, which can be deduced for instance by appropriately specializing $(7.9)'$ and $(7.13)'$  of \cite[Chapter VI]{MacMainBook}: 

\smallskip
\noindent
Projection:
\begin{equation}\label{eq:projdefn}
\wl[q=0] = s_\lambda(X_n).
\end{equation}
Branching:
\begin{equation}\label{eq:branch}
W_\lambda(X_n; q) = \sum_{\mu \prec \lambda} \prod_{1 \leq i <n} \qbinom{\lambda_i-\lambda_{i+1}}{\lambda_i-\mu_i} \; W_\mu(X_{n-1}; q) \; x_n^{|\lambda|-|\mu|}
\end{equation}
Both these also follow readily from \eqref{eq:schurpoly} and \eqref{eq:fermpop}.

\begin{definition}
    We define {\em combinatorial projection} to be the map
\begin{align}
  \pr: \pop &\to \GT \notag\\
  \pr(T,\Lambda) &= T\label{eq:prdef}
\end{align}
and {\em combinatorial branching} to be the map
\begin{align}
  \br: \pop &\to \bigsqcup_{\mu \prec \lambda} \pop[\mu]\notag\\
  \br(T,\Lambda) &= (T^\dag, \Lambda^\dag)\label{eq:brdef}
\end{align}
where the GT pattern $T^\dag$ is obtained from $T \in \GT$ by deleting its bottom-most row, and $\Lambda^\dag$ is obtained from $\Lambda$ by deleting the overlays $\Lambda_{ij}$ with $j=n-1$.
%, i.e., those straddling the bottom two rows of $T$. 
\end{definition}
\subsection{Box complementation} In addition to the combinatorial projection and branching maps, $\pop$ is endowed with another important map, which we term {\em box complementation}.\footnote{We are indebted to Reiner-Shimozono \cite{reiner-shimozono-percentage} for this terminology, though it has a slightly different connotation in their context.} Observe that for a partition $\pi = (\pi_1 \geq \pi_2 \geq \cdots \geq \pi_k \geq 0)$ fitting into a $k \times \ell$ rectangle, we may consider its complement in this rectangle, defined by $\pi^c = (\ell-\pi_k \geq \ell-\pi_{k-1} \geq \cdots \geq \ell-\pi_1)$ (Figure~\ref{fig:boxcomp}).
\begin{figure}
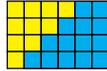

    \centering
    \ydiagram[*(cyan)]
  {4+2,3+3,2+4,1+5}
*[*(yellow)]{6,6,6,6}
    \caption{Box-complementation.}
    \label{fig:boxcomp}
\end{figure}
Clearly:
\begin{equation}\label{eq:pipic}
|\pi| + |\pi^c| = k\ell.
\end{equation}
Now, for $(T,\Lambda) \in \pop$, define
\begin{equation}\label{eq:bcdef}
\bcomp(T,\Lambda)=(T,\Lambda^c) \in \pop
\end{equation}
where for each $i,j$, $(\Lambda^c)_{ij}$ is defined to be the complement of $\Lambda_{ij}$ in its bounding rectangle of size $\NE_{ij}(T) \times \SE_{ij}(T)$. 
We note that $(T,\Lambda)$ and its box complement have the same $x$-weight, but not necessarily the same $q$-weight, since $|\Lambda|\neq |\Lambda^c|$ in general. Following  \cite{RRV-CLpop}, we define  \begin{equation}\label{eq:areadef}
\area(T) := \sum_{1 \leq i \leq j < n} \NE_{ij}(T)\SE_{ij}(T), \;\;\;\text{ for } T \in \GT. 
\end{equation}
It follows from \eqref{eq:pipic}, \eqref{eq:bcdef} and \eqref{eq:areadef} that for all $(T,\Lambda) \in \pop$,
\begin{equation}\label{eq:areaprop}
    |\Lambda| + |\Lambda^c| = \area(T).
\end{equation}
The definition of the combinatorial branching map in \eqref{eq:brdef} readily implies the following equality of maps from $\pop \to \bigsqcup_{\mu \prec \lambda} \pop[\mu]$:
\begin{equation}\label{eq:boxbr}
  \br \circ \bcomp = \bcomp \circ \br
  \end{equation}
\subsection{Chari-Loktev bases of local Weyl modules}\label{sec:locweyl}
We recall from the introduction that the \qw polynomial $\wl$ is the graded character of the {\em local Weyl module} $W_{\mathrm{loc}}(\lambda)$ - a finite-dimensional cyclic module of the current Lie algebra $\mathfrak{sl}_n[t] :=\mathfrak{sl}_n\otimes \mathbb{C}[t]$ \cite{ChariIon-BGG,ChariLoktev-original}. Further, $\pop$ indexes a special basis of this module with branching properties akin to the classical Gelfand-Tsetlin basis of irreducible $\mathfrak{gl}_n$-modules \cite{ChariLoktev-original,RRV-CLpop}. We briefly recall the definition of this basis of Chari and Loktev, in the formulation of \cite[\S 4.6]{RRV-CLpop}. We will use this later on in Proposition~\ref{prop:clbasis-csf-Et}.

Let $\nminus$ denote the set of strictly lower triangular matrices in $\sln$, and let $E_{pq} \in \nminus$  be the elementary matrix (with a $1$ in precisely the $(p,q)^{th}$ entry for $n \geq p>q \geq 1$). Consider the Lie algebra $\nminus[t]=\nminus \otimes \complex[t]$. To each $\mfp = (T,\Lambda) \in \pop$, we associate a monomial  $\CL(\mfp) \in \uea\nminus[t]$ (the universal enveloping algebra of $\nminus[t]$) as follows: 
\begin{enumerate}
    \item For $1 \leq i \leq j <n$, let $\lij(\mfp)$ be the product of $\left(E_{j+1, \,i} \otimes t^p\right)$ as $p$ runs over all parts of the partition  $\Lambda_{ij}$, i.e.,
    if $\Lambda_{ij} = (\Lambda_{ij}(1) \geq \Lambda_{ij}(2) \geq \cdots \geq \Lambda_{ij}(k) \geq 0)$ with $k=\NE_{ij}(T)$, then 
    \begin{equation}\label{eq:lij1}
        \lij(\mfp) = \prod_{m=1}^{\NE_{ij}(T)} \left(E_{j+1, \,i} \otimes t^{\,\Lambda_{ij}(m)}\right).
        \end{equation}
    We note that since the factors commute, the order is immaterial.
    \item For $1 \leq j <n$, define \begin{equation}\label{eq:lij2} \lij[j](\mfp) := \prod_{i=1}^j \lij(\mfp).\end{equation} Here again the factors commute.
    \item Finally, define the (ordered) product:
    \begin{equation}\label{eq:clbasis-mon}
\CL(\mfp) := \lij[1](\mfp) \,\lij[2](\mfp)\cdots \lij[n-1](\mfp).
    \end{equation}
\end{enumerate}
We call $\CL(\mfp) \in \uea\nminus[t]$ the {\em Chari-Loktev monomial} associated to $\mfp \in \pop$. We have the following key theorem \cite{ChariLoktev-original, RRV-CLpop}: 
\begin{thm}\label{thm:clthm} (Chari-Loktev)
Let $\lambda$ be a partition with at most $n$ nonzero parts. Let $w_\lambda$ denote the cyclic generator of the local Weyl module $\wloc$ of $\curr$. Then the set
\[ \{ \CL(\mfp) \,w_\lambda: \mfp \in \pop \} \]
is a homogeneous basis of $\wloc$. Further, if $\mfp=(T,\Lambda)$, the element $CL(\mfp) \,w_\lambda$ has $\sln$-weight $x^T$ and $q$-grade $|\Lambda|$.  \qed
\end{thm}

Now, equation~\eqref{eq:branch} specialized at $x_n=1$ can be viewed as an identity relating characters of local Weyl modules of $\mathfrak{sl}_n[t]$ and $\mathfrak{sl}_{n-1}[t]$. Chari and Loktev lift \eqref{eq:branch} to the level of modules, showing that  $W_{\mathrm{loc}}(\lambda)$ when restricted to $\mathfrak{sl}_{n-1}[t]$ admits a filtration whose successive quotients are of the form $W_{\mathrm{loc}}(\mu)$ for $\mu \prec \lambda$ \cite[Cor 2.2.5]{ChariLoktev-original}. Further, their graded multiplicities are precisely given by the products of $q$-binomial coefficients that appear in \eqref{eq:branch}.

\section{Projection and branching for Column strict fillings}\label{sec:probrancsf}
We seek to construct natural bijections between $\csf$ and $\pop$ which explain the equality of the monomial expansions \eqref{eq:wlinvquinv} and \eqref{eq:fermpop} for $v=\inv, \quinv$. In addition to preserving $x$- and $q$-weights, we require our bijections to be compatible with the combinatorial projection and branching maps. Towards this end, we first define these latter maps in the setting of column strict fillings. 

\subsection{Projection: rowsort} Given $F \in \csf$, let $\rsort(F)$ denote the filling obtained from $F$ by sorting entries of each row in ascending order from left to right (Figure~\ref{fig:rsort}).
\begin{figure}
    \centering
    \ytableausetup{mathmode, smalltableaux}  
    $F=$  \ytableaushort{4121,543,65} \qquad $\rsort(F)=$  \ytableaushort{1124,345,56}
    \caption{An example of the rowsort operation.}
    \label{fig:rsort}
\end{figure}
In light of the following easy lemma, we think of $\rsort$ as the projection map in the CSF setting.
\begin{lemma}
  If $F \in \csf$, then $\rsort(F) \in \SSYT\cong \GT$.
\end{lemma}
 \begin{proof}
 For each $1 \leq i \leq \lambda'_1$, let $R_i$ denote the multiset (of cardinality $\lambda_i$) comprising the entries of the $i^{th}$ row of $F$. 
 Observe that the entry in the cell $(i,j)$ of $\rsort(F)$ is the $j^{th}$ smallest element of $R_i$. 
 Now consider the map $R_i \to R_{i-1}$ (for $i>1$) which sends $F(u) \mapsto F(\up[u])$ for each cell $u$ in the $i^{th}$ row of $F$.  Since $F$ is a column-strict filling of partition shape, this defines an injective, strictly decreasing map of multisets. This implies in particular that the $j^{th}$ smallest element of $R_i$ is greater than the $j^{th}$ smallest element of $R_{i-1}$ for each $1 \leq j \leq |R_i|$, in other words that $\rsort(F) \in \csf$. Since the entries of each row of $\rsort(F)$ are already weakly increasing, we conclude $\rsort(F) \in \SSYT$. 
 \end{proof}

\subsection{Towards branching: the elementary splice operation}\label{sec:elsplice}
A strictly increasing sequence $a=(a_1 < a_2 < \cdots < a_p)$ of positive integers (possibly empty) will be termed a {\em column tuple} with $\len(a):=p \geq 0$. It will be convenient to define $a_0:=0$ for any column tuple $a$. 

Let $0 \leq k < \ell $ and suppose 
\[\sigma=(\sigma_1 < \sigma_2 < \cdots < \sigma_{k}) \text{ and } \tau=(\tau_1 < \tau_2 < \cdots < \tau_{\ell})\] 
are column tuples of lengths $k$ and $\ell$ respectively. 
We now define the {\em elementary splice operation} on such pairs $(\sigma, \tau)$. Consider the set $\{1 \leq i \leq k+1: \sigma_{i-1} < \tau_i\}$. Since $\sigma_0=0$, it is non-empty. Let
\begin{equation}\label{def:mdef}
m := \max \,\{1 \leq i \leq k+1: \sigma_{i-1} < \tau_i\}.
\end{equation}
Define $\splice(\sigma,\tau) := (\barsig,\bartau)$ where
\begin{equation}\label{eq:spliceeq}
  \barsig_i = \begin{cases} \sigma_i & 1\leq i <m \\ \tau_i &m \leq i \leq \ell \end{cases} \;\;\;\;\;\;\; \text{ and } \;\;\;\;\;\; \bartau_i = \begin{cases} \tau_i & 1\leq i <m \\ \sigma_i & m \leq i \leq k \end{cases}
  \end{equation}
i.e., $\barsig,\bartau$ are obtained by swapping certain suffix portions of $\sigma, \tau$. By the choice of $m$, it directly follows that $\barsig$ is a column tuple and that
\begin{equation}\label{eq:splicecor}
\sigma_i \geq \tau_{i+1} > \tau_i \text{ for } m \leq i \leq k.
\end{equation}
This latter fact ensures that $\bartau$ is a column tuple as well.

We also have $\len(\barsig) = \len(\tau) = \ell$ and $\len(\bartau)=\len(\sigma) = k$. For instance, when  \ytableausetup{mathmode, smalltableaux}  \[(\sigma,\tau) = (\; \ytableaushort{1,5} \; ,\, \ytableaushort{2,3,4}\;), \text{ we get } (\barsig,\bartau) = (\; \ytableaushort{1,3,4}\; ,\, \ytableaushort{2,5}\;).\]

We extend the definition of $\splice$ to arbitrary pairs of column tuples $(\sigma,\tau)$ as follows:
\begin{enumerate}
    \item If $\len(\sigma) < \len(\tau)$, then $\splice(\sigma,\tau)$ is given by \eqref{eq:spliceeq}.
    \item If $\len(\sigma) > \len(\tau)$, then \begin{equation}\label{eq:extsplice}\splice(\sigma, \tau):= \omega(\splice (\tau,\sigma)) = \omega \circ \splice \circ \,\omega (\sigma,\tau)
    \end{equation}
    where $\omega$ is the ``swap" operator, mapping $(a,b) \mapsto (b,a)$.  
    \item If $\len(\sigma) = \len(\tau)$, then \begin{equation}\label{eq:equalitysplice}
    \splice(\sigma,\tau):=(\sigma, \tau).
    \end{equation}
\end{enumerate}

\begin{lemma} \label{lem_splice}
 Suppose $0 \leq k<\ell$ and $\sigma=(\sigma_1 < \sigma_2 < \cdots < \sigma_k)$ and $\tau=(\tau_1 < \tau_2 < \cdots < \tau_{\ell})$ are column tuples of length $k$ and $\ell$ respectively. If $\splice(\sigma,\tau) = (\barsig,\bartau)$, then:
 \[\max\,\{1 \leq i \leq k+1: \sigma_{i-1} < \tau_{i}\} = \max\, \{1 \leq i \leq k+1: \bartau_{i-1} < \barsig_{i}\}\] and $\splice(\bartau,\barsig)= (\tau,\sigma)$.
 \begin{proof}
     Let $m$ be the maximum of  $\{1 \leq i \leq k+1: \sigma_{i-1} < \tau_{i}\}$. It follows from the definition of $\splice$ and \eqref{eq:splicecor} that 
     \[ \bartau_{i-1} = \sigma_{i-1} \geq \tau_i =  \barsig_{i} \text{ for all } i > m.\] 
     Since $\barsig_{m} = \tau_{m}$ and $\bartau_{m-1} = \tau_{m-1}$, we have $\barsig_{m} > \bartau_{m-1}$. This implies that $m$ is also the maximum of $\{1 \leq i \leq k+1: \bartau_{i-1} < \barsig_{i}\}$. Therefore, $\splice(\bartau,\barsig)= (\tildetau,\tildesig)$, where
\[ \tildetau_i = \begin{cases} \bartau_i & 1\leq i <m \\ \barsig_i &m \leq i \leq \ell \end{cases} \;\;\;\;\;\;\; \text{ and } \;\;\;\;\;\; \tildesig_i = \begin{cases} \barsig_i & 1\leq i <m \\ \bartau_i &m \leq i \leq k \end{cases}\]
It follows from \eqref{eq:spliceeq} that $\tildesig =\sigma $ and $\tildetau = \tau$.
 \end{proof}
\end{lemma}
As a direct consequence of \eqref{eq:extsplice} and Lemma~\ref{lem_splice}, we obtain:
\begin{corollary}\label{cor:corone}
$\splice$ is an involution on the set of all pairs of column tuples, i.e., $(\splice)^2$ is the identity.
\end{corollary}

\subsection{Splice and CSFs of column composition shapes} We will now extend the definition of the splice operation to CSFs. Given a CSF $F$ of partition shape and two of its adjacent columns, we can consider the new CSF obtained by applying the splice operator to that pair of columns, while leaving the other columns of $F$ unchanged. This produces a CSF whose shape is not necessarily a partition. 

We call the diagrams obtained by permuting the columns of a partition as {\em column compositions} (we allow empty columns). Now, let $\lambda$ be a partition; the columns of $\dg(\lambda)$ have lengths $\lambda_i' \; (i=1, \cdots, \lambda_1)$. Let $\Comp(\lambda)$ denote the set of all column composition shapes obtained by permuting the columns of $\lambda$. 
Let $\gamma \in \Comp(\lambda)$. If $d:=\lambda_1$, then $\dg(\gamma)$ has $d$ columns; we let $\gamma'_i$ denote the length of the $i^{th}$ column of $\dg(\gamma)$. Define $\csf[\gamma]$ to be the set of column strict fillings $F:\dg(\gamma) \to [n]$.

\begin{definition}\label{def:spliceoncsf}
Let $\lambda$ be a partition with $d:=\lambda_1$ and let $1 \leq i < d$. The $i^{th}$ splice operation $\mathbb{S}_i$  is the map 
\[ \mathbb{S}_i: \bigsqcup_{\gamma \in \Comp(\lambda)} \csf[\gamma] \to \bigsqcup_{\gamma \in \Comp(\lambda)} \csf[\gamma].\]
defined as follows: given $F \in \csf[\gamma]$, $\mathbb{S}_i(F)$ is obtained by replacing the $i^{th}$ and $(i+1)^{th}$ columns of $F$ by the result of applying the splice operator of \S\ref{sec:elsplice} to that pair, while leaving the other columns of $F$ unchanged. 
\end{definition}
%%%%%%%%%%%%%%%%%%%%%%%%%%%%
\begin{example}
Let $n=5$ and $\lambda= (8,6,3,1,0)$; thus $\lambda'=(4,3,3,2,2,2,1,1)$. Let $F$ be the CSF in Figure~\ref{fig:example1}; the column lengths of $\dg(F)$ form a permutation of those of $\dg(\lambda)$. Shown below are the images of $F$ under the $i^{\,th}$ splice operation $\spi(F)$ for $i=1, 3$. The cells of  $F$ and  $\spi(F)$ are colour-coded to highlight how the $i^{th}$ and $(i+1)^{th}$ columns of $F$ are transformed by $\spi$.
\end{example}
\begin{figure}[h]
\begin{center}
\begin{tikzpicture}
\draw (3.40,7.5)node {\ytableausetup{mathmode,boxsize=1.2em}
$F=\;$\begin{ytableau}
  *(red)1 & *(lime) 1 &   *(magenta) 2 &*(green) 1   &  *(orange) 2 &  *(pink) 3 &  *(yellow) 2 & *(cyan) 5\\
   *(red) 2 & *(lime) 3 & *(magenta) 3 & *(green) 2   &  *(orange) 3 &  \none &*(yellow)3 \\
 *(red) 3 & \none & \none &*(green) 4   & \none & \none & *(yellow)4 \\
  \none &\none &\none &*(green) 5
\end{ytableau} };
\draw (0.00,5.00)node {\ytableausetup{mathmode,boxsize=1.2em}
  $\spi[1](F)=\;$\begin{ytableau}
  *(red)1 & *(lime) 1 &   *(magenta) 2 &*(green) 1   &  *(orange) 2 &  *(pink) 3 &  *(yellow) 2 & *(cyan) 5\\
   *(lime) 3 & *(red) 2 & *(magenta) 3 & *(green) 2   &  *(orange) 3 &  \none &*(yellow)3 \\
 \none &*(red) 3 &  \none &*(green) 4   & \none & \none & *(yellow)4 \\
  \none &\none &\none &*(green) 5
  \end{ytableau}\;};
\draw (6.80,5.00)node {\ytableausetup{mathmode,boxsize=1.2em}$\spi[3](F)=\;$\begin{ytableau}
  *(red)1 & *(lime) 1 &   *(magenta) 2 &*(green) 1   &  *(orange) 2 &  *(pink) 3 &  *(yellow) 2 & *(cyan) 5\\
   *(red) 2 & *(lime) 3 & *(magenta) 3 & *(green) 2   &  *(orange) 3 &  \none &*(yellow)3 \\
 *(red) 3 & \none & *(green) 4  &\none & \none & \none & *(yellow)4 \\
 \none &\none &*(green) 5  
\end{ytableau}};
\end{tikzpicture}
\end{center}
\caption{}\label{fig:example1}
\end{figure}

%%%%%%%%%%%%%%%%%%%%%%%%%%%%%%
The following properties are immediate consequences of the definition: (i) The shape of $\mathbb{S}_i(F)$ is $s_i(\dg(\gamma))$ where $s_i$ denotes the simple transposition in $S_d$ which acts on diagrams of column compositions by swapping the columns $i, i+1$. (ii) The multiset of entries of each given row of $F$ coincides with that of the corresponding row of $\mathbb{S}_i(F)$. 

It follows from Corollary~\ref{cor:corone} that 
\begin{equation}\label{eq:involution}
\spi^2=1 \text{ (the identity map) for } 1 \leq i <d.
\end{equation}
It is also clear from Definition~\ref{def:spliceoncsf} that 
\begin{equation}\label{eq:commuting}
\spi \spi[j] = \spi[j] \spi \text{ whenever } 1 \leq i,j <d \text{ with } |i-j|>1,
\end{equation}
since in this case $\spi$ and $\spi[j]$ act on disjoint pairs of columns. 

It turns out that the $\spi$ also satisfy the braid relations
\begin{equation}\label{eq:braid}
\spi \spi[i+1] \spi = \spi[i+1] \spi[i] \spi[i+1]
\end{equation}
for $1 \leq i < (d-1)$, but we will not require the full force of this assertion in what follows. So we content ourselves with proving the following special case of immediate interest, since it has a very short proof. We defer the general case and its consequences to an upcoming sequel.
\begin{lemma}\label{lem:elembraid}
Let $\gamma$ be a column composition with $d$ columns. If $1 \leq i  \leq d-2$ is such that some two among $\gamma'_i, \gamma'_{i+1}, \gamma'_{i+2}$ are equal, then $\spi \spi[i+1] \spi (F) = \spi[i+1] \spi[i] \spi[i+1] (F)$ for all $F \in \csf[\gamma]$.
\end{lemma} 
\begin{proof}
     We consider three cases:
    \begin{enumerate}
    \item $\gamma'_i=\gamma'_{i+1}$: Equation \eqref{eq:equalitysplice} implies that $\spi(F) = F$. Further the $(i+1)^{th}$ and $(i+2)^{th}$ columns of $\spi\spi[i+1](F)$ are of equal length. Appealing to \eqref{eq:equalitysplice} again, we conclude that $\spi \spi[i+1] \spi (F) = \spi \spi[i+1](F) = \spi[i+1]\spi \spi[i+1](F)$.     
    \item $\gamma'_{i+1}=\gamma'_{i+2}$: This follows from the argument of Case 1 by swapping the roles of $\spi$ and $\spi[i+1]$.
    \item $\gamma'_{i}=\gamma'_{i+2}$: The $(i+1)^{th}$ and $(i+2)^{th}$ columns of $\spi(F)$ are equal and we obtain $\spi \spi[i+1] \spi(F)  = \spi^2(F) = F$ by Corollary~\ref{cor:corone}. Likewise, $\spi[i+1] \spi \spi[i+1]  = \spi[i+1]^2(F) = F$. 
    \end{enumerate}
\end{proof}

\begin{corollary}\label{cor:corab}
Let $a \neq b$ be natural numbers.
Let $\lambda$ be a partition with $d:=\lambda_1$. Suppose $\lambda_i' \in \{a, b\}$ for all $i$, i.e., the columns of $\dg(\lambda)$ have one of two possible lengths. Then the operators $\{\spi: 1 \leq i <d\}$ on $\bigsqcup_{\gamma \in \Comp(\lambda)} \csf[\gamma]$ satisfy the relations \eqref{eq:involution}, \eqref{eq:commuting}  and \eqref{eq:braid}. 
\end{corollary}
\begin{proof}
We only need to establish the braid relations \eqref{eq:braid}. This follows from Lemma~\ref{lem:elembraid} since every $\gamma \in \Comp(\lambda)$ now satisfies the hypothesis of this lemma.
\end{proof}
Let $d$ be the number of columns of $\dg(\lambda)$. Let $s_i$ denote the simple transposition $(i \,\; i+1) \in S_d$ for $1 \leq i <d$. Since these $s_i$ are generators of $S_d$, with defining relations $s_i^2 =1$ and
\begin{equation}\label{eq:commbraid}
s_i s_j = s_j s_i \; \text{ if } |i-j|>1 \text{ and } s_is_{i+1}s_i = s_{i+1}s_is_{i+1},
\end{equation}
we obtain:
\begin{corollary}\label{cor:cortwo}
Let $\lambda$ be a partition with at most two distinct column lengths and let $d:=\lambda_1$. 
    The map $s_i \mapsto \spi$ extends to a well-defined action of $S_d$ on $\bigsqcup_{\gamma \in \Comp(\lambda)} \csf[\gamma]$. \qed
\end{corollary}
If we let $s_i \cdot F$ denote $\spi(F)$, then we clearly have $\dg(s_i \cdot F)=s_i(\dg(F))$ for all $F \in \bigsqcup_{\gamma \in \Comp(\lambda)} \csf[\gamma]$. 

As mentioned earlier, the braid relations and thereby the assertion of Corollary~\ref{cor:cortwo} hold more generally - for all partitions $\lambda$. The proof of this fact and its generalizations will form part of a subsequent paper. Corollary~\ref{cor:cortwo} suffices for our present purposes, and will allow us to show well-definedness of the branching map for column strict fillings. We now proceed to establish this.

\subsection{Branching: delete-and-splice}\label{sec:dsplice}
We now formulate the notion of branching in the CSF model. 
\begin{definition}\label{def:dsplice} Given $F \in \csf$,  its {\em delete-and-splice rectification} (denoted $\dsplice(F)$)  is the result of applying the following algorithm to $F$:
\begin{enumerate}
    \item delete all cells in $F$ containing the entry $\boxed{n}$ and let $\fdag$ denote the resulting filling. While its column entries remain strictly increasing, $\fdag$ may no longer be of partition shape. 
    \item Let $\sigma^{(j)}\, (j \geq 1)$ denote the column tuple obtained by reading the $j^{\text{\,th}}$ column of $\fdag$ from top to bottom.  If $\fdag$ is not of partition shape, there exists $j \geq 1$ such that $\len(\sigma^{(j+1)}) > \len(\sigma^{(j)})$.
        Choose any such $j$ and replace $\fdag$ by $\spi[j]\fdag$. This brings the shape of $\fdag$ one step closer to being a partition.
    \item If the shape of $\fdag$ is a partition, STOP. Else go back to step 2.
\end{enumerate}
It is clear that this process terminates and finally produces a CSF of partition shape (filled by numbers between $1$ and $n-1$), which we denote $\dsplice(F)$.\qed
\end{definition}

\begin{proposition}\label{prop:dsplice-prop}
With notation as above: 
\begin{enumerate} 
\item $D:=\dsplice(F)$ is independent of the intermediate choices of $j$ made in step 2 of the algorithm. 
\item $\rsort(D)$ is obtained from $\rsort(F)$ by deleting the cells containing the entry $n$.
\item If $\mu$ and $\lambda$ are the shapes of $D$ and $F$ respectively, then $\mu \prec \lambda$. 
\end{enumerate}
\end{proposition}

\noindent
In other words, $\dsplice$ is a well-defined map:
\[ \dsplice: \csf \to \bigsqcup_{\mu \prec \lambda}\csf[\mu]\]
where for $\mu \prec \lambda$, we let $\csf[\mu]$ denote the set of column strict fillings $F: \dg(\mu) \to [n-1]$.
We consider $\dsplice$ to be the combinatorial branching map in the CSF context. It has the remarkable property of being  compatible with the natural branching map $\br$ of the POP setting, as will be established in Theorem~\ref{thm:mainthm}.
 We note that while each $\splice$ operation is ``local'' -  affecting two adjacent columns at a time - the end result $\dsplice(F)$ can have a fair bit of ``intermixing'' amongst columns of $F$, as Example~\ref{eg:mixing} shows.
%%%%%%%%%%%%%%%%%%%%%%%%%%%%%%%%%%%%%%%%%
\begin{example}\label{eg:mixing}
In Figure~\ref{fig:mix}, we take $n=6$ and $\lambda= (10,7,5,2,0)$. Let $F \in \csf$ and let $\fdag$ be the CSF of column composition shape obtained from $F$ by deleting the boxes containing the entry $\boxed{6}$. Then $\dsplice(F)$ is a CSF of shape $\mu= (9,7,4,2,0)$ with entries $\leq 5$. The columns of $F$, $\fdag$, and $\dsplice(F)$  are colour-coded in Figure~\ref{fig:mix} to show the intermingling under the $\dsplice$ operation.
\end{example}
\begin{figure}[t]
\begin{center}
\begin{tikzpicture}
\draw (0.00,7.5)node {\ytableausetup{mathmode,boxsize=1.2em}
$F=\;$\begin{ytableau}
  *(green) 1  &  *(red)1 & *(cyan) 2 & *(purple) 1 &  *(orange) 2 &  *(pink) 3 &  *(yellow) 2 & *(magenta) 4 & *(olive) 6 &*(lime) 5\\
  *(green) 2  & *(red) 2 & *(cyan) 3 &  *(purple) 2 &  *(orange) 3 &  *(pink) 6 &*(yellow)3& *(magenta) 5 \\
  *(green) 4 & *(red) 3 &*(cyan)  4 &  *(purple) 6 & *(orange)5 \\
  *(green) 6&*(red)4&*(cyan)  5
  \end{ytableau}};
\draw (6.6,7.5)node {\ytableausetup{mathmode,boxsize=1.2em}
  $\fdag=\;$ \begin{ytableau}
  *(green) 1  &  *(red)1 & *(cyan) 2 & *(purple) 1 &  *(orange) 2 &  *(pink) 3 &  *(yellow) 2 &*(magenta) 4 &\none & *(lime) 5 \\
  *(green) 2  & *(red) 2 & *(cyan) 3 &  *(purple) 2 &  *(orange) 3 & \none &*(yellow)3 & *(magenta) 5\\
  *(green) 4 & *(red) 3 &*(cyan)  4&  \none & *(orange)5\\
  \none &*(red)4&*(cyan)  5
\end{ytableau}};
\draw (3.2,4.80)node {\ytableausetup{mathmode,boxsize=1.2em}$\dsplice(F)=\,$\begin{ytableau}
  *(green) 1  &  *(red)1  & *(cyan) 2 &  *(purple) 1  &  *(orange) 2 &  *(yellow) 2  &  *(pink) 3 & *(magenta) 4 & *(lime) 5 \\
  *(green) 2  & *(red) 2 & *(cyan) 3  &  *(purple) 2 &  *(orange) 3 & *(yellow)3 & *(magenta) 5\\
 *(red) 3& *(green) 4 & *(cyan) 4  &  *(orange)5\\
 *(red)4&*(cyan)  5
\end{ytableau}};
\end{tikzpicture}
\end{center}
\caption{The dsplice operation intermixes columns.}\label{fig:mix}
\end{figure}
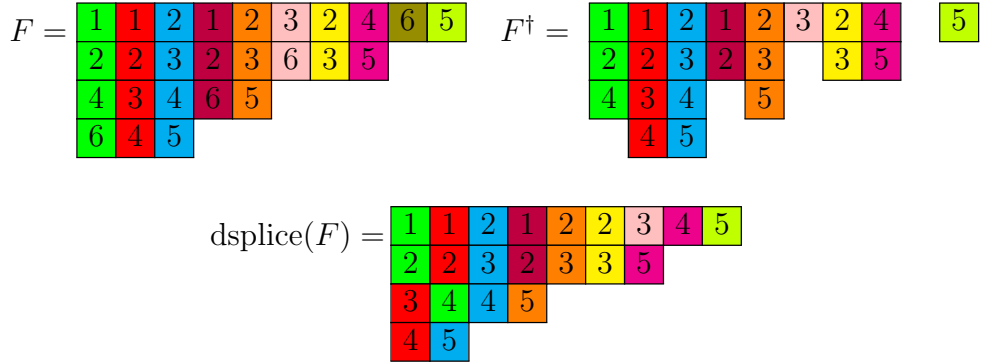
%%%%%%%%%%%%%%%%%%%%%%%%%%%%%%%%%%%%%%%%%%%
\subsection{} We now prove Proposition~\ref{prop:dsplice-prop}. 
Let $d:=\lambda_1$ be the number of columns of $F$.
Let $\gamma$ denote the (column composition) shape of $\fdag$. Let $\mu$ be the unique partition such that $\dg(\mu)$ is in the $S_d$-orbit of $\dg(\gamma)$, i.e., $\mu$ is obtained by rearranging the columns of $\gamma$ in descending order of length. Let $w^\gamma\in S_d$ be the unique element of minimal length
 such that $w^\gamma(\dg(\gamma)) = \dg(\mu)$. It is clear from Definition~\ref{def:dsplice} that iteration of step 2 of the algorithm therein results in the choice of $j_1, j_2, \cdots, j_r$ with the following properties:
\begin{enumerate}
    \item $s_{j_r}s_{j_{r-1}} \cdots s_{j_1}$ is a reduced word in $S_d$,
    \item $w^\gamma = s_{j_r}s_{j_{r-1}} \cdots s_{j_1}$.
    \end{enumerate}
Note that when the algorithm terminates, $\dsplice(F)$ is defined to be  
\begin{equation}\label{eq:dspspi}
\spi[j_r]\spi[j_{r-1}]\cdots \spi[1](\fdag).
\end{equation}
Let $\supp w^\gamma$ denote the set comprising $1 \leq i <d$ such that $s_i$ occurs in some (and hence in all) reduced decomposition(s) of $w^\gamma$. 
If the intermediate choices of $j$ in the algorithm were made differently, resulting in a sequence $(j'_i)$, then  $s_{j'_r}s_{j'_{r-1}} \cdots s_{j'_1}$ would be another reduced word for $w^\gamma$ (and hence of the same length as the previous one). 
The well-known theorem of Tits on the word problem (see \cite[Theorem 3.3.1]{bjorner-brenti}) asserts that any two reduced words may be transformed one into the other by using only the  commutation and braid relations \eqref{eq:commbraid}. In light of \eqref{eq:commuting} and \eqref{eq:dspspi}, $\dsplice(F)$ would be well-defined if we establish that the operators $\{\spi: i \in \supp w^\gamma\}$ satisfy   \eqref{eq:braid} on an appropriate subset \[\mathcal{Y} \subset \bigsqcup_{\pi \in \Comp(\mu)} \csf[\pi]\] which contains $\fdag$.

To find such $\mathcal{Y}$, we begin with the case when $\lambda$ is a rectangular partition, say $\lambda = (d, d, \cdots, d)$ ($p$ times). Since $F$ has shape $\lambda$ and each column of $F$ can contain at most one $\boxed{n}$, the columns of $\fdag$ are of length $p$ or $p-1$. Thus $\mu$ satisfies the hypothesis of Corollary~\ref{cor:corab}, which shows that we may take $\mathcal{Y} =\bigsqcup_{\pi \in \Comp(\mu)} \csf[\pi]$ in this case.

For the general case, we decompose $\dg(\lambda)$ into a concatenation of rectangular partition diagrams as shown in Figure~\ref{fig:decrectpart}. 
\begin{figure}[h]
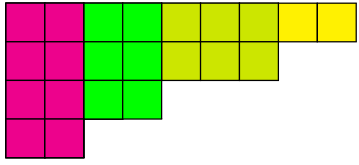

    \centering
    \ydiagram[*(magenta)]
  {2,2,2,2}
*[*(green)]{4,4,4,2}
*[*(Mycolor6)]{7,7,3,2}
*[*(yellow)]{9,7,3,2}
\caption{Decomposing a partition into  rectangular sub-partitions.}
    \label{fig:decrectpart}
\end{figure}
Formally, let $0=d_0 < d_1 < d_2 < \cdots < d_q=d$ be the unique indices such that (i) $\lambda'_i = \lambda'_j$ for all $d_{p-1} < i, j \leq d_p$ for $1 \leq p \leq q$ and (ii) $\lambda'_{d_1} > \lambda'_{d_2} > \cdots > \lambda'_{d_q}$. Let $\mathbb{D}:=\{d_p: 1 \leq p \leq q\}$.

Since $\lambda$ is a partition and $\lambda'_i - \gamma'_i=0$ or $1$ for all $i$, it follows that 
\begin{equation}\label{eq:gamprime}
\gamma'_i \geq \gamma'_j \text{ if } i \leq d_{p}< j \text{ for all } 1 \leq p <q.
\end{equation}
If $j_1, j_2, \cdots, j_r$ is a sequence of $j$'s generated by iteration of step 2 of the algorithm, it follows from \eqref{eq:gamprime} that $j_i \notin \mathbb{D}$ for all $1 \leq i \leq r$. Thus,
\[\supp w^\gamma \subseteq \mathbb{D}^c :=\{1 \leq i <d: i \notin \mathbb{D}\},\]
i.e.,  $w^\gamma \in Y$, the subgroup of $S_d$ generated by the $\{ s_i: i \in \mathbb{D}^c\}$. It is clear that $Y$ is isomorphic to the Young subgroup $S_{d_1} \times S_{d_2} \times \cdots \times S_{d_p}$.
Let $Y\cdot\mu$ denote the $Y$-orbit of $\mu$ and define
\[ \mathcal{Y} :=\bigsqcup_{\pi \in Y\cdot\mu} \csf[\pi]\]
Since $\gamma \in Y\cdot \mu$, we have $\fdag \in \mathcal{Y}$.
We now prove that the operators $\{\spi: i \in \mathbb{D}^c\}$ on $\mathcal{Y}$ satisfy the braid relations \eqref{eq:braid}. If $i, i+1 \in \mathbb{D}^c$, it follows that there exists $d_p \in \mathbb{D}$ such that $d_{p-1} <i<i+1<i+2 \leq d_p$. But since $\lambda'_j$ has the same value for all $d_{p-1} <j \leq d_p$, the corresponding $\gamma'_j$ can only take one of two values $\{\lambda'_j, \lambda'_j-1\}$. Therefore the pair $\gamma, i$ satisfies the hypothesis of Lemma~\ref{lem:elembraid}. 

More generally, given $\pi \in Y\cdot \mu = Y\cdot \gamma$, the columns of $\pi$ in each rectangular block are permutations of the columns in the corresponding block of $\gamma$, i.e., $(\pi'_j: d_{p-1} <j \leq d_p)$ is a permutation of $(\gamma'_j: d_{p-1} <j \leq d_p)$. Thus in light of the preceding argument, $\pi, i$ also satisfies the hypothesis of Lemma~\ref{lem:elembraid}. Proposition~\ref{prop:dsplice-prop}, part (1) now follows by appealing to Lemma~\ref{lem:elembraid}.

For part (2), we recall our earlier observation that the action of $\spi$ on a CSF (of column composition shape) leaves the multisets of entries of each row invariant. From \eqref{eq:dspspi}, it follows that the corresponding rows of $D:=\dsplice(F)$ and $\fdag$ have the same entries (possibly permuted). Since $\fdag$ is obtained from $F$ by deleting the cells containing $\boxed{n}$, it follows readily that $\rsort(D)$ is likewise obtained from $\rsort(F)$ by deleting the cells $\boxed{n}$. 

Part (3) follows from general facts about semistandard Young tableaux. The cells containing $\boxed{n}$ form a horizontal strip \cite[Chapter I]{MacMainBook} in $\rsort(F)$. This is equivalent to the condition that deleting them leaves a shape $\mu$ which interlaces $\lambda$. \qed

\section{The main results}\label{sec:mainresults}

\begin{thm}\label{thm:mainthm}
  For any $n \geq 1$ and any partition $\lambda: \lambda_1 \geq \lambda_2 \geq \cdots \geq \lambda_n \geq 0$ with at most $n$ nonzero parts, there exist two bijections $\psi_{\inv}$ and $\psi_{\quinv}$ from $\csf$ to $\pop$ with the following properties:

  \medskip
  \noindent
    1. If $\psi_v(F) = (T,\Lambda)$, then $x^F = x^T$ and $v(F) = |\Lambda|$, for $v=\inv$ or $\quinv$.
    
\medskip
\noindent
2. The following diagrams commute ($v=\inv$ or $\quinv$):
\begin{enumerate}
  \item[(A)]
\[\begin{tikzcd}
	\csf && \pop \\
	& \GT
	\arrow["{\psi_v}", from=1-1, to=1-3]
	\arrow["\rsort"'{pos=0.4}, from=1-1, to=2-2]
	\arrow["\pr"{pos=0.4}, from=1-3, to=2-2]
\end{tikzcd}\]

\item[(B)] \[\begin{tikzcd}
	\csf &&& \pop \\
	\\
	\displaystyle\bigsqcup_{\mu \prec \lambda}\csf[\mu] &&& \displaystyle\bigsqcup_{\mu \prec \lambda}\pop[\mu]
	\arrow["{\psi_v}", from=1-1, to=1-4]
	\arrow["\dsplice"', from=1-1, to=3-1]
	\arrow["{\psi_v}"', from=3-1, to=3-4]
	\arrow["\br", from=1-4, to=3-4]
\end{tikzcd}\]
\end{enumerate}

\medskip
\noindent
3. The two bijections are related via the commutative diagram:
\[\begin{tikzcd}
	& \csf \\
	\pop && \pop
	\arrow["{\psi_{\mathrm{inv}}}"', from=1-2, to=2-1]
	\arrow["{\psi_{\mathrm{quinv}}}", from=1-2, to=2-3]
	\arrow["\bcomp", from=2-1, to=2-3]
\end{tikzcd}\]\qed
\end{thm}

%\noindent
To summarize, $\psi_{\mathrm{inv}}$ and $\psi_{\mathrm{quinv}}$ acting on a CSF produce partition overlaid patterns  with the same underlying GT pattern, but with complementary overlays. These bijections are compatible with the natural projection and branching maps, and preserve $x$-weights and the appropriate $q$-weights (inv or quinv). 

As in Proposition~\ref{prop:dsplice-prop}, note the slight abuse of notation in part 2(B) above: for $\mu \prec \lambda$, $\csf[\mu]$ denotes the set of column strict fillings $F: \dg(\mu)\to [n-1]$ (rather than $[n]$). Theorem~\ref{thm:mainthm}, with the exception of part 2(B), can also be formulated in the setting of $q$-Whittaker functions in infinitely many variables.

We record the following corollaries:
\begin{corollary}
Let $T \in \GT$ and let $\rsort^{-1}(T) = \{F \in \csf: \rsort(F)=T\}$ be the fiber of $\rsort$ over $T$. Then
\begin{enumerate}
\item \begin{equation}\label{eq:wtqinter}
\displaystyle\sum_{F \in \rsort^{-1}(T)} q^{\,\inv(F)} = \displaystyle\sum_{F \in \rsort^{-1}(T)} q^{\,\quinv(F)} = \wt_q(T).
\end{equation}
\item \[\inv(F) + \quinv(F) = \area(T)\] is  constant for $F \in \rsort^{-1}(T)$.
\end{enumerate}\qed
\end{corollary}
\begin{proof}
The first assertion follows from part 2(A) of Theorem \ref{thm:mainthm} and equation~\eqref{eq:popfiber}. The second is a consequence of part 3 of Theorem \ref{thm:mainthm} and \eqref{eq:areaprop}.
\end{proof}
A very different interpretation of $\wt_q(T)$ in terms of flags of subspaces compatible with nilpotent operators appears in \cite[Theorem 5.8(i)]{karpthomas}.

In \cite{AMM}, Ayyer-Mandelshtam-Martin asked for an explicit bijection on $\scf$ which interchanges the $\inv$ and $\quinv$ statistics while preserving $\maj$ and leaving the entries of each row of the filling invariant. We describe this bijection on $\csf$ (the subset with maximal $\maj$), thereby partially answering their question (see also \cite{ratheesh-ijpam} and \cite{jin-lin}):
\begin{corollary}\label{cor:bijamm}
The map $\Omega: \psi_{inv}^{-1} \circ \psi_{\quinv} = \psi_{inv}^{-1} \circ \, \bcomp \circ \, \psi_{\inv}: \csf \to \csf$ is an involution satisfying $\inv(\Omega(F)) = \quinv(F)$ for all $F \in \csf$. Further, $\rsort(F) = \rsort(\Omega(F))$.
\end{corollary}
\begin{proof}
    The first assertion follows from part (1) of Theorem~\ref{thm:mainthm}, while the second assertion follows from part (2A).
\end{proof}

The statistics $zcount$, $\overline{zcount}$ of the following proposition are key ingredients in the proof of Theorem~\ref{thm:mainthm}, and are described in greater detail in sections~\ref{sec:cwzc}, \ref{sec:zcb}. They turn out to be the precise notions required to formulate the Chari-Loktev basis in the language of CSFs. 
\begin{proposition}\label{prop:clbasis-csf-Et}
    There exist non-negative integer valued maps $\zcount$ and $\zcb$ on $\dg(\lambda) \times \csf$ satisfying the following properties:
    \begin{enumerate}
        \item For all $F \in \csf$: \begin{align*}
        \sum_{c \in \dg(\lambda)} \zcount(c,F) &= \quinv(F)\\
        \sum_{c \in \dg(\lambda)} \zcb(c,F) &= \inv(F).
        \end{align*}
        \item If $F, F' \in \csf$ with $\rsort(F) = \rsort(F')$, then for all $c \in \dg(\lambda)$:
        \[\zcount(c,F) + \zcb(c,F) = \zcount(c,F') + \zcb(c,F').\] 
        \item Let $F \in \csf$ and $\mfp_v := \psi_v(F) \in \pop$ for $v \in \{\inv, \quinv\}$. If $\mathfrak{b}_v(F) :=\CL(\mfp_v)$ denotes the corresponding Chari-Loktev monomials, then:  
\begin{align}
\mathfrak{b}_{\quinv}(F) &= \displaystyle\prod_{c \in \dg(\lambda)}^{\rightharpoonup} E_{F(c),\,i(c)} \otimes t^{\,\zcount(c,F)} \label{eq:bquinv}\\
\mathfrak{b}_{\inv}(F) &= \displaystyle\prod_{c \in \dg(\lambda)}^{\rightharpoonup} E_{F(c), \,i(c)} \otimes t^{\,\zcb(c,F)} \label{eq:binv}
\end{align}
where for each $c \in \dg(\lambda)$, $i(c)$ denotes the row number in which $c$ occurs, and $\displaystyle\prod^{\rightharpoonup}$ means that the cells of $\dg(\lambda)$ are enumerated in weakly increasing order of $F(c)$ (any such enumeration gives the same final value in the products above). \qed
    \end{enumerate}
\end{proposition}
Putting together Theorem~\ref{thm:clthm} and Proposition~\ref{prop:clbasis-csf-Et}, we obtain two natural monomial bases of $\wloc$: 
\begin{corollary}\label{cor:clbasiscsf}
    Let $w_\lambda$ denote the cyclic generator of the local Weyl module $\wloc$ of $\curr$. Then, for $v \in \{\inv, \quinv\}$, the set
\[ \{ \mathfrak{b}_v(F) \,w_\lambda: F \in \csf\} \]
is a homogeneous basis of $\wloc$, with the $q$-grade of $\mathfrak{b}_v(F)$ being $v(F)$ and its $\sln$-weight being $x^F$. \qed
\end{corollary}
Sections \ref{sec:qtrip}--\ref{sec:compf} are devoted to the construction of the $\psi_v$ and the proof of Theorem~\ref{thm:mainthm} and Proposition~\ref{prop:clbasis-csf-Et}. The explicit constructions of $\psi_v$ and $\psi_v^{-1}$ for $v=\inv, \quinv$ make the bijection $\Omega$ of Corollary~\ref{cor:bijamm} effectively computable. They also lead to the elegant descriptions of the bases in Corollary~\ref{cor:clbasiscsf}.

\section{quinv-triples}\label{sec:qtrip}
In this section, we define $\psi_{\quinv}$ and establish part 1 of Theorem 1 for $v=\quinv$.
\subsection{Cellwise zcounts}\label{sec:cwzc}
Given a partition, or more generally a column composition $\lambda$, the augmented diagram $\dgh(\lambda)$ is $\dg(\lambda)$ together with one additional cell below the last cell in each column (see Figure~\ref{fig:zcfig}). Given $u \in \dg(\lambda)$, let $\down[u] \in \dgh(\lambda)$ denote the cell immediately below $u$ in its column.

\begin{definition}\label{def:qtrip}
    Let $\lambda$ be a column composition. Given $F \in \csf$, a {\em quinv-triple} in $F$ is a triple of cells $(x,y,z)$ in $\dgh(\lambda)$ such that
    \begin{enumerate}
        \item[(i)] $x, z \in \dg(\lambda)$ and $z$ is to the right of $x$ in the same row, \item[(ii)] $y = \down[x]$,
        \item[(iii)] $F(x) < F(z) < F(y)$, where we set $F(y) = \infty$ if $y$ lies outside $\dg(\lambda)$.
    \end{enumerate}
\end{definition}

When $\lambda$ is a partition, it is easy to see that the quinv-triples considered in \cite[Definition 2.3]{AMM} for $F \in \scf$ reduce to this description when $F$ is a CSF rather than a general filling. We thus obtain from \cite{AMM} that for a partition $\lambda$ and $F \in \csf$:
\begin{equation}\label{eq:quinvtotal}
\quinv(F) = \text{ number of quinv-triples in } F. 
\end{equation}
We define a function {\em zcount} which tracks the contributions of individual cells of $\dg(\lambda)$ to $\quinv(F)$ as follows.
\begin{definition}\label{def:zcdef}
Let $\lambda$ be a column composition. Given $F \in \csf$ and a cell $c \in \dg(\lambda)$, define 
$\zcount(c, F)$ to be the number of quinv-triples  
 $(x,y,z)$  in $F$ with $z=c$.
\end{definition}

\begin{figure}
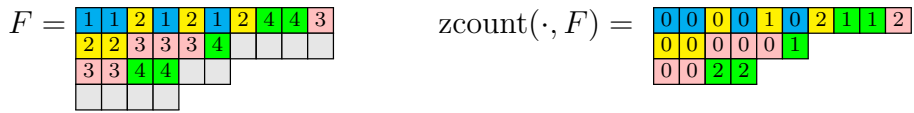

  \begin{center}
  \ytableausetup{mathmode, smalltableaux}
$F=\;$\begin{ytableau}
  *(cyan) 1  & *(cyan) 1 & *(yellow) 2 & *(cyan) 1 & *(yellow) 2 & *(cyan) 1 & *(yellow) 2 & *(green) 4 & *(green) 4 & *(pink) 3\\
  *(yellow) 2  & *(yellow) 2 & *(pink) 3 & *(pink) 3 & *(pink) 3 & *(green) 4 &*(light-gray) &*(light-gray) &*(light-gray) &*(light-gray) \\
  *(pink) 3  & *(pink) 3 & *(green) 4 & *(green) 4 &*(light-gray) &*(light-gray) \\
  *(light-gray) &*(light-gray) &*(light-gray) &*(light-gray) 
  \end{ytableau} \qquad \quad
 $\zcount(\cdot,F)=\;$ \begin{ytableau}
  *(cyan) 0  & *(cyan) 0 & *(yellow) 0 & *(cyan) 0 & *(yellow) 1 & *(cyan) 0 & *(yellow) 2 & *(green) 1 & *(green) 1 & *(pink) 2\\
  *(yellow) 0  & *(yellow) 0 & *(pink) 0 & *(pink) 0 & *(pink) 0 & *(green) 1\\
  *(pink) 0  & *(pink) 0 & *(green) 2 & *(green) 2
\end{ytableau}
\caption{Here $F \in \csf$ for $\lambda=(10,6,4,0)$ and $n=4$. Cells of $F$ are coloured according to their entries. The gray cells are the extra cells in the augmented diagram $\dgh(\lambda)$. On the right are cellwise zcount values. Here $\quinv(F)=12$.}
  \label{fig:zcfig}
\end{center}
\end{figure}

\subsection{}
Let $\lambda$ be a partition and $F \in \csf$. We clearly have from \eqref{eq:quinvtotal} that
\begin{equation}\label{eq:qwtquinv}
  \sum_{c \in \dg(\lambda)} \zcount(c,F) = \quinv(F).
\end{equation}
We next group together cells of the filling $F$ row-wise according to the entries they contain. More precisely,  let
\[\cells(i,j,F) = \{c \in \dg(\lambda): c \text{ is in the } i^{th} \text{ row and } F(c) = j+1\}\]
for $1 \leq i \leq j+1 \leq n$.
Figure~\ref{fig:zcfig} shows an example, with these groups colour-coded in each row.
It readily follows from \S\ref{sec:fermfor} that
\begin{equation}\label{eq:cellcard}
  |\cells(i,j,F)| = \NE_{ij}(T),  \text{ where } T = \rsort(F).
\end{equation}
The next proposition brings the SE differences also into play:
\begin{proposition}\label{prop:zcounts}
  Let $\lambda$ be a partition, $F \in \csf$ and $T = \rsort(F)$. Fix $1 \leq i \leq j+1 \leq n$.
  \begin{enumerate}
  \item If $c \in \cells(i,j,F)$, then \begin{equation}\label{eq:zcse} \zcount(c,F) \leq \SE_{ij}(T).
  \end{equation}
  \item If $c, d \in  \cells(i,j,F)$ with $c$ lying to the right of $d$, then 
  \[\zcount(c,F) \geq \zcount(d,F).\]
  \item Further, equality holds in \eqref{eq:zcse} above for all $i, j$ and all cells $c \in \cells(i,j,F)$ iff $F=T$.
  \end{enumerate}
\end{proposition}
  \begin{proof}
Let $c \in \cells(i,j,F)$. Since $F(c)=j+1$,  a $\quinv$-triple $(x,y,z)$ with $z=c$ satisfies $F(x) \leq j$ and $F(y) \geq j+2$ (with $F(y)=\infty$ if  
$y$ lies outside $\dg(\lambda)$). Therefore $\zcount(c,F)$ is precisely the number of pairs of cells $(x,y)$ satisfying: 
\begin{enumerate}
    \item[(i)] $x \in \dg(\lambda)$ is in the $i^{\text{th}}$ row with $F(x) \leq j$ and 
    \item[(ii)] $y=\down[x]$  and $F(y) \geq j+2$.
    \item[(iii)] $x$ occurs to the left of $c$.
\end{enumerate}
The set of cells $x$ satisfying the first property is precisely 
\begin{equation}\label{eq:eqc}
C:=\bigcup_{k<j} \cells(i,k,F).
\end{equation}
Since $F$ and $T=\rsort(F)$ have the same entries in each row, it follows from the notation for GT patterns in \S\ref{sec:fermfor} that $|C| = T^j_i$.

Now consider 
\begin{equation}\label{eq:eqcprime}
C' = \bigcup_{k\leq \,j} \cells(i+1,k,F),
\end{equation}
which as before has cardinality $|C'| = T^{j+1}_{i+1}$. For each $y \in C'$, the cell $\up[y]$ above it must contain an entry $\leq j$, since $F$ is a CSF. A moment's thought allows us to conclude  from this fact that a pair $(x,\down[x])$ satisfies properties (i), (ii) above iff 
\[ x \in C - \upop(C') = C - \{\up[y]: y \in C'\}.\]
This set has cardinality 
\begin{equation}\label{eq:cupc}
|C - \upop(C')| = T^{j}_{i} - T^{j+1}_{i+1} = \SE_{ij}(T).
\end{equation}
Finally, incorporating property (iii) as well, $\zcount(c,F)$ is precisely the cardinality of the subset of $C - \upop(C')$ comprising cells which occur to the left of $c$. This proves part (1) of the proposition. This argument also establishes the following condition for equality: $\zcount(c,F) = \SE_{ij}(T)$ iff
\begin{equation}\label{eq:eqimp}
\text{ every cell in } C - \upop(C') \text{ occurs to the left of } c.
\end{equation}

 To prove (2), let $c, d \in  \cells(i,j,F)$ with $c$ lying to the right of $d$. Since $F(c) = F(d)$, it follows from Definition~\ref{def:qtrip} that if $(x,y,d)$ is a quinv-triple, then so is $(x,y,c)$. Thus $\zcount(c,F) \geq \zcount(d,F)$.

For (3), suppose that $\zcount(c,F) = \SE_{ij}(T)$ for all $i, j$, and all cells $c \in \cells(i,j,F)$, To prove $F=T$, it suffices to show that the entries in each row of $F$ weakly increase from left to right. We proceed by induction; this holds for the $n^{\text{th}}$ row of $F$ since all its entries are $\boxed{n}$. Now suppose that the weakly increasing condition holds for all rows strictly below the $i^{\text{th}}$ row, we claim that it also holds for the $i^{\text{th}}$ row. If not, then there exists $j \geq 1$ and two cells $c \in \cells(i,j,F)$ and $d \in \cells(i,j-1,F)$ such that $c$ occurs to the left of $d$. In view of \eqref{eq:eqimp}, $d \notin \{\up[y]: y \in C'\}$ (notation as in proof of part (1)). Thus the cell $\down[d]$ must contains the entry $j+1$. On the other hand, $\down[c]$ contains an integer greater than $F(c) = j+1$. This contradicts the assumption that the entries in the $(i+1)^{\text{th}}$ row are weakly increasing from left to right. This completes the induction argument. 
Conversely if $F=T$, it is clear that \eqref{eq:eqimp} holds for $c \in \cells(i, j, F)$. 
\end{proof}

  \subsection{Definition of $\psi_{\quinv}$} We now have all the ingredients in place to define $\psi_{\quinv}$. Let $\lambda$ be a partition, $F \in \csf$ and $T =\rsort(F)$. For each $1 \leq i \leq j+1 \leq n$, consider the sequence
  \begin{equation}\label{eq:lamij}
    \Lambda_{ij} = (\zcount(c,F): \; c \in \cells(i,j,F) \text{ traversed right to left in row } i).
    \end{equation}
In Figure~\ref{fig:zcfig}, this amounts to reading the entries of a fixed colour from right to left in a given row of $\zcount(\cdot, F)$.
  By Proposition~\ref{prop:zcounts}, $\Lambda_{ij}$  is a weakly decreasing sequence, with values bounded above by $\SE_{ij}(T)$. Further \eqref{eq:cellcard} states that this sequence has exactly $\NE_{ij}(T)$ terms. Thus, $\Lambda_{ij}$ may be viewed as a partition fitting into the $\NE_{ij}(T) \times \SE_{ij}(T)$ rectangle. Since $\SE_{ij} =0$ for $i=j+1$, $\Lambda_{ij}$ is the zero sequence in this case. We drop the pairs $(j+1,j)$ to conclude that if $\Lambda = (\Lambda_{ij}: 1 \leq i \leq j <n)$, then $(T, \Lambda) \in \pop$. We define \[\psi_{\quinv}(F) := (T, \Lambda).\] 
  Clearly, $x^F = x^T$ and \eqref{eq:qwtquinv} implies $\quinv(F) = |\Lambda|$, establishing (1) of Theorem~\ref{thm:mainthm} for $v=\quinv$. The remaining properties of $\psi_{\quinv}$ will be established in section~\ref{sec:compf}.

\section{inv and refinv triples}\label{sec:invreftrip}
\subsection{}
We now turn to the definition of $\psi_{\inv}$. While we may anticipate doing this via a modification of the foregoing arguments, replacing quinv-triples with Haglund-Haiman-Loehr's {\em inv-triples} \cite{HHL-I}, that turns out not to work out-of-the-box. 
We consider instead:
\begin{definition}\label{def:refinvdef}
Let $\lambda$ be a column composition. A {\em refinv-triple} (or ``reflected inv-triple'') for $F \in \csf$
is a triple $(x,y,z)$ in $\dgh(\lambda)$ where:
\begin{enumerate}
\item[(i)] $x, z \in \dg(\lambda)$ with $z$ to the left of $x$ in the same row,
\item[(ii)] $y = \down[x]$,
\item[(iii)]  $F(x) < F(z) < F(y)$, where $F(y):=\infty$  if $y \not\in \dg(\lambda)$.
\end{enumerate}
Let $\refinv(F)$ denote the number of refinv-triples of $F$. \qed
\end{definition}

Refinv triples were originally defined for partition shapes\footnote{He doesn't name them though - our choice of name is motivated by the close parallel to inv triples.} by Zelevinsky in his Russian translation of the first edition of Macdonald's monograph  - see \cite[(AZ.6), Appendix to Chapter II]{MacMainBook}, \cite[\S 2.2]{Kirillov_newformula} - predating their inv and quinv cousins (for CSFs) by several years. 
Figure~\ref{fig:triplesconfig} exhibits the configurations of $\quinv, \inv$ and $\refinv$ triples.\footnote{We do not formally define the notion of inv-triples here since we do not use it, directing the interested reader to \cite{HHL-I} instead.} 

We now have the following key relationship:
\begin{proposition}\label{prop:refinv}
Let $\lambda$ be a partition and $F \in \csf$. Then $\inv(F) = \refinv(F)$.
\end{proposition}
\begin{figure}
  \begin{center}
    \begin{tikzpicture}
      \ytableausetup{mathmode, smalltableaux}
      \draw(0,0)node {\ydiagram[*(yellow)]{1,1}};
      \draw(2,0.2)node{\ydiagram[*(yellow)]{1}};
      \draw[dashed](0.3,0.2)--(1.7,0.2);
    \end{tikzpicture}
    \qquad
        \begin{tikzpicture}
      \ytableausetup{mathmode, smalltableaux}
      \draw(0,0)node {\ydiagram[*(cyan)]{1,1}};
      \draw(2,-0.2)node{\ydiagram[*(cyan)]{1}};
      \draw[dashed](0.3,-0.2)--(1.7,-0.2);
    \end{tikzpicture}
\qquad
        \begin{tikzpicture}
      \ytableausetup{mathmode, smalltableaux}
      \draw(2,0)node {\ydiagram[*(orange)]{1,1}};
      \draw(0,0.2)node{\ydiagram[*(orange)]{1}};
      \draw[dashed](0.3,0.2)--(1.7,0.2);
    \end{tikzpicture}
        \caption{(left to right) Configuration of quinv, inv and refinv triples. }
        \label{fig:triplesconfig}
  \end{center}
\end{figure}
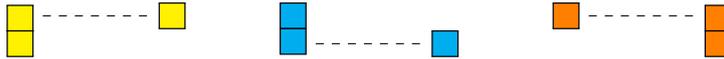

\begin{proof}
Our starting point is \cite[\S 2]{HHL-I} which shows that $\inv(F)$ coincides with the number of inv-triples of $F$ (in fact for all fillings $F$). We repurpose their argument below to prove that
\begin{equation}\label{eq:refinveq}
\refinv(F) = |\Inv(F)| - \sum_{u \in \Des(F)} \coarm(\up[u])
\end{equation}
for $F \in \csf$. Observe that 
\begin{equation}\label{eq:coarmup}
\coarm(\up[u])= \coarm(u) \text{ for } u \in \Des(F).
\end{equation}
Now, the descent set $\Des(F)$ is a union of rows; it comprises all cells of $\dg(\lambda)$ except those in the first row. This implies that \[\sum_{u \in \Des(F)} \arm(u) = \sum_{u \in \Des(F)} \coarm(u). \] 
Equations~\eqref{eq:inveq}, \eqref{eq:refinveq} and \eqref{eq:coarmup} complete the proof. It thus only remains to establish \eqref{eq:refinveq}.
Following \cite{HHL-I}, we fix $F \in \csf$ and define for $u, v \in \dgh{\lambda}$:
\begin{equation}
I(u, v)=
\begin{cases}
1 & \text{if   $  F(u) > F(v) $ }\\
0 & \text{if   $  F(u) \leq F(v) $. }
\end{cases} \notag  
\end{equation}  
Let $\mathcal{R}$ denote the set of triples  $(x,y,z)$ of cells in $\dgh(\lambda)$ satisfying properties (i) and (ii) of Definition~\ref{def:refinvdef}; clearly refinv-triples form a subset of $\mathcal{R}$. 

The column strictness of $F$ implies that for any triple $(x,y,z) \in \mathcal{R}$, $I(y,x)=1 $, and $I(y, z) + I(z, x)  \neq 0$, in other words 
\[ \ihat:=I(y,z) + I(z, x) - I(y,x) = 0 \text{ or } 1.\]
It is further evident that $\ihat=1$ if and only if $(x,y,z)$ is a refinv-triple, i.e.,
\begin{equation}\label{eq:refI} \refinv(F) = \sum_{(x,y,z) \in \mathcal{R}} \ihat 
\end{equation}
We decompose  $\mathcal{R} = \mathcal{T} \sqcup \mathcal{T}^\dag$ where $\mathcal{T}$ (resp. $\mathcal{T}^\dag$) comprises triples $(x,y,z) \in \mathcal{R}$ with $y \in \dg(\lambda)$ (resp. $y \in \dgh(\lambda) - \dg(\lambda)$).   
Using equation~\eqref{eq1.0.6}, we observe the following facts:
\begin{enumerate}
    \item if $(x,y,z) \in \mathcal{T}$, then  $I(y,z) + I(z,x)$ is the contribution to $|\Inv(F)|$ from the pairs $(y, z)$ and $(z, x)$.
    \item if $(x,y,z) \in \mathcal{T}^\dag$, then $I(y,z) + I(z, x) - I(y,x) = I(z, x)$, and this is the contribution to $|\Inv(F)|$ from the pair $(z, x)$. 
\end{enumerate}
Now, any pair of cells $(u, v)$ in $\Inv(F)$ can be uniquely ``completed" to a triple $(x,y,z) \in \mathcal{R}$, with $(x,y,z)=(\up[u], u,v)$ if $u, v$ are in consecutive rows and $(x,y,z) = (v, \down[v], u)$  if they are in the same row.  This implies 
\begin{align} 
\label{eqrefinv 1.1}
    |\Inv(F)| &= \sum_{(x, y, z) \in \mathcal{T}} \left(I(y,z) + I(z,x)\right) +  \sum_{(x, y, z) \in \mathcal{T}^\dag}  \ihat
\end{align}
From \eqref{eq:refI} and \eqref{eqrefinv 1.1}, we obtain
\[  |\Inv(F)| - \refinv(F) = \sum_{(x, y, z) \in \mathcal{T}} I(y,x) = |\mathcal{T}|\]
We may rewrite the sum  above as
\[ \sum_{(x, y, z) \in \mathcal{T}} I(y,x) = \sum_{y \in \Des(F)} \sum_{\substack{z \\(\up[y], y, z) \in \mathcal{T}}} 1 = \sum_{y \in \Des(y)} \coarm(y^{\uparrow}).\] 
This finishes the proof of \eqref{eq:refinveq}.
\end{proof}

\subsection{$\zcb$, $\zcount$ and the proof of the main theorem}\label{sec:zcb}
By analogy with Definition~\ref{def:zcdef}, given a column composition $\lambda$, $F \in \csf$ and $c \in \dg(\lambda)$, define
\begin{equation}\label{eq:zcbdef}
\zcb(c, F) := \text{ number of refinv-triples } (x,y,z) \text{ in } F \text{ with } z=c.
\end{equation}
In light of Proposition~\ref{prop:refinv}, it is clear that
\begin{equation}\label{eq:qwtinv}
\sum_{c \in \dg(\lambda)} \zcb(c,F) = \inv(F)
\end{equation}
for $\lambda$ a partition and $F \in \csf$. We have the following relation between $\zcb$ and $\zcount$:
\begin{proposition}\label{prop:zcandzcb}
  Let $\lambda$ be a partition, $F \in \csf$ and $T=\rsort(F)$. Let $1 \leq i \leq j+1 \leq n$ and $c \in \cells(i,j,F)$. Then
  \[\zcount(c,F) + \zcb(c,F) = \SE_{ij}(T).\]
\end{proposition}
\begin{proof}
Let $c \in \cells(i,j,F)$. We recall the argument that proved part (1) of Proposition~\ref{prop:zcounts}. Consider the sets $C, C'$ from \eqref{eq:eqc}, \eqref{eq:eqcprime} and let \[\mathfrak{C} := C - \upop(C').\]
As shown in the proof of Proposition~\ref{prop:zcounts}, $\zcount(c,F)$ is number of $x \in \mathfrak{C}$ which occur to the left of $c$. A similar argument, appealing now to the definition of $\zcb$ in \eqref{eq:zcbdef} shows that $\zcb(c,F)$ is the number of $x \in \mathfrak{C}$ which occur to the right of $c$. Thus, the sum $\left(\zcount(c,F) + \zcb(c,F)\right)$ is exactly the cardinality of $\mathfrak{C}$. Equation~\eqref{eq:cupc} finishes the proof.
\end{proof}

The following proposition readily follows from Propositions \ref{prop:zcounts}  and  \ref{prop:zcandzcb}.
\begin{proposition}\label{prop:barzcounts}
  Let $\lambda$ be a partition,  $F \in \csf$ and $T = \rsort(F)$. Fix $1 \leq i \leq j+1 \leq n$.
  \begin{enumerate}
  \item If $c \in \cells(i,j,F)$, then $\zcb(c,F) \leq \SE_{ij}(T)$.
  \smallskip\item If $c, d \in  \cells(i,j,F)$ with $c$ lying to the left of $d$, then \[\zcb(c,F) \geq \zcb(d,F).\]
    \item Further,  $\zcb(c,F)=0$ for all $i, j$ and all cells $c \in \cells(i,j,F)$ iff $F=T$.
  \end{enumerate}
 \end{proposition}

\subsection{}
We may now define $\psi_{\inv}$ following the template of $\psi_{\quinv}$. Given a partition $\lambda$ and $F \in \csf$, let $T =\rsort(F)$. For each $1 \leq i \leq j < n$,  consider the sequence:
  \begin{equation}\label{eq:lamijbar}
  \overline{\Lambda}_{ij} = (\zcb(c,F): \; c \in \cells(i,j,F) \text{ traversed left to right in row } i).
  \end{equation}
   Recall also the definition of the partition $\Lambda_{ij}$ from \eqref{eq:lamij}. It follows from Propositions~\ref{prop:zcounts} and \ref{prop:zcandzcb} that $\overline{\Lambda}_{ij}$ is the box complement of $\Lambda_{ij}$ in the $\NE_{ij}(T) \times \SE_{ij}(T)$ rectangle. Letting $\overline{\Lambda} = (\overline{\Lambda}_{ij}: 1 \leq i \leq j<n)$, we define \[\psi_{\inv}(F) = (T,\overline{\Lambda}).\] As in the case of $\quinv$, we have $x^F = x^T$, and  $\inv(F) = |\overline{\Lambda}|$ by \eqref{eq:qwtinv}. This proves part (1) of Theorem~\ref{thm:mainthm} for $v=\inv$.

\section{Completing the proof of Theorem~\ref{thm:mainthm}}\label{sec:compf}
\subsection{}\label{sec:splice-br-compat} We continue using the notations of the previous two sections. The first assertion of Theorem~\ref{thm:mainthm} has already been established in sections~\ref{sec:qtrip},\ref{sec:invreftrip} following the definitions of $\psi_{\quinv}$ and $\psi_{\inv}$. Since by definition $\pr (\psi_v(F)) = T$ for $v=\inv, \quinv$,  Part (2A) of Theorem~\ref{thm:mainthm} follows. Part (3) of Theorem~\ref{thm:mainthm}  follows from the fact that $\Lambda$ and $\overline{\Lambda}$ are box complements of each other in the appropriate rectangles.
We now show that the diagrams in part (2B) of Theorem~\ref{thm:mainthm} are commutative - this will follow from a careful analysis of each elementary splice step of the $\dsplice$ map.

Let $\lambda$ be a partition and $F \in \csf$. Let $\fdag$ denote the filling obtained from $F$ by deleting cells containing $\boxed{n}$ as in \S\ref{sec:dsplice}. We view $\dg(\fdag)$ as a subset of $\dg(F)$.
Notice that we have 
\begin{equation}\label{eq:cfcfdag}
\cells(i,j,F) = \cells(i,j,\fdag) \text{ for } 0 \leq j < n-1,
\end{equation}
and for $c \in \cells(i,j,F)$, we have
\begin{equation}\label{eq:zcffdag}
  \zcount(c, F) = \zcount(c, \fdag)
  \text{ and } \zcb(c, F) = \zcb(c,\fdag)
  \end{equation}

Let $\sigma^{(j)}\, (j \geq 1)$ denote the column tuple obtained by reading the $j^{\text{\,th}}$ column of $\fdag$ from top to bottom. Let $j$ be any fixed column index such that $\len(\sigma^{(j+1)}) > \len(\sigma^{(j)})$ (and thus necessarily  $\len(\sigma^{(j+1)}) = \len(\sigma^{(j)}) +1$). The elementary splice step replaces the pair of columns $(\sigma^{(j)}, \sigma^{(j+1)})$ in $\fdag$ by $\splice(\sigma^{(j)}, \sigma^{(j+1)})$; we denote the resulting filling by $\ftilde$. 

There is a bijective correspondence $c \mapsto \ctilde\;$ between the cells of $\dg(\fdag)$ and $\dg(\ftilde)$ such that $\fdag(c) = \ftilde(\ctilde)$ described as follows: (a) if $c$ is not in the $j^{\text{\,th}}$ or $(j+1)^{\text{\,th}}$ columns of $\fdag$, then $\ctilde = c$ (b) if $c$ lies in one of these two columns, then $\ctilde$ is  the cell of $\ftilde$ to which the entry in $c$ moves under the splice operation (i.e., only the cells corresponding to the suffix portions of these columns in \eqref{eq:spliceeq} are swapped). We further note that this bijection between cells of $\dg(\fdag)$ and $\dg(\ftilde)$ extends to a bijection between the cells of the corresponding augmented diagrams.

Observing that the shapes of $\fdag$ and $\ftilde$ are column compositions and recalling the general notions of section \ref{sec:cwzc}, we now state the following key lemma.
\begin{lemma}\label{lem:keylem} 
With notation as above, for all cells $c$ of $\dg(\fdag)$ we have  \[\zcount(c, \fdag\,) = \zcount(\,\ctilde, \ftilde\,).\]
\end{lemma}
\begin{proof}
  The proof necessitates the verification of a handful of cases. We will establish a bijective correspondence between the sets of quinv-triples counted by $\zcount(c, \fdag)$ and $\zcount(\,\ctilde, \ftilde\,)$. 
  
  Suppose $c$ occurs in row $r$. Consider the cells of the augmented diagram of $\fdag$ with coordinates \[(r,j),\; (r,j+1),\; (r+1,j), \text{ and } (r+1, j+1)\] and let their entries be $\alpha_1, \alpha_2, \beta_1, \beta_2$ respectively (see Figure~\ref{fig:a1a2b1b2}). Let $\gamma$ be the entry in cell $c$.

\begin{figure}
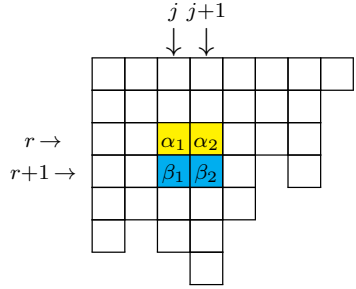

\begin{center}
\ytableausetup{mathmode, boxframe=normal, boxsize=1em}
\begin{ytableau}
\none & \none & \none & \none & \none[\scriptstyle j] & \none[\scriptstyle j+1] \\
\none & \none & \none & \none & \none[\downarrow] & \none[\downarrow] \\
 \none & \none & {} &  &   &  &  & & & \\
\none&\none &  & 
     &  &  &  &  &  \\
 \none[\scriptstyle r\,\rightarrow] & \none &  &  & *(yellow)\scriptstyle{\alpha_1} & *(yellow)\scriptstyle{\alpha_2} & 
&  &  \\ 
\none[\scriptstyle r+1\,\rightarrow] & \none &  &  & *(cyan)\scriptstyle{\beta_1} & *(cyan)\scriptstyle{\beta_2} & & \none & \\
\none & \none &  & 
     &  &  &   \\
   \none & \none & & \none & & \\
    \none & \none&  \none & \none & \none & 
  \end{ytableau} 
  \end{center}
  \caption{A schematic diagram showing a column composition $\fdag$ and the choice of $\alpha_i, \beta_i$.}\label{fig:a1a2b1b2}
  \end{figure}

 \medskip
  \noindent
  {\bf Case 1:} $c$ is not in columns $j$ or $j+1$. In this case, $c =\ctilde$. Consider quinv-triples in $\fdag$ of the form $(x,y,c)$. If $y=\ctilde[y]$ (and hence $x=\ctilde[x]$) or if $x \neq \ctilde[x]$ (and hence $y \neq \ctilde[y]$), it is clear that $(\ctilde[x], \ctilde[y], c)$ is a quinv-triple for $\ftilde$, and conversely. 
  
  It only remains to consider the case $x = \ctilde[x], y \neq \ctilde[y]$. Under this assumption, \eqref{def:mdef} and \eqref{eq:splicecor} imply that
    \begin{equation}\label{eq:mainineq}
      \beta_1 > \beta_2 > \alpha_1.
  \end{equation}

 Consider the triples $t_i$ in $\fdag$ formed by the cells containing $\alpha_i, \beta_i, \gamma$ for $i=1, 2$. Likewise in $\ftilde$, consider the triples $\widetilde{t}_i$ formed by the corresponding spliced cells, i.e., with entries $\alpha_1, \beta_2, \gamma$ and $\alpha_2, \beta_1, \gamma$ respectively.   Let $I_1 = (\alpha_1, \beta_1), I_2=(\alpha_2, \beta_2), \widetilde{I}_1 =(\alpha_1, \beta_2), \widetilde{I}_2=(\alpha_2, \beta_1)$ be the open intervals in $\mathbb{Z}$ with the indicated bounds. Observe that $t_k$ (resp. $\widetilde{t}_k$) is a quinv-triple for $\fdag$ (resp. $\widetilde{F}$) iff $\gamma \in I_k$ (resp. $\widetilde{I}_k$).

  We claim that the number of quinv-triples among the $t_i$ is the same as that among the $\widetilde{t}_i$ ($i=1,2$).
  The inequalities \eqref{eq:mainineq} readily imply the following equalities of sets:
  \[ I_1 \cup I_2 = \widetilde{I}_1 \cup \widetilde{I}_2 \text{ and } I_1 \cap I_2 = \widetilde{I}_1 \cap \widetilde{I}_2. \]

In particular, this implies that 
\[ |\{1 \leq k \leq 2: \gamma \in I_k\}| = |\{1 \leq k \leq 2: \gamma \in \widetilde{I}_k\}|\]
establishing our claim, and finishing the proof in this case.

\medskip
\noindent
{\bf Case 2:} $c$ is in the $j^{th}$ column and $c=\ctilde$. Then it clear that $(x,y,c)$ is a quinv-triple of $\fdag$ if and only if it is a quinv triple of $\ftilde$.

\medskip
\noindent
{\bf Case 3:}
  $c$ is in the $j^{th}$ column and $c\neq \ctilde$. We have $\gamma = \alpha_1 \geq \beta_2$ (from \eqref{eq:spliceeq}). If $(x,y,c)$ is a quinv-triple in $\fdag$, then $x,y$ lie in a column strictly to the left of the $j^{th}$ column; thus $(x,y,\ctilde)$ will be a quinv-triple in $\ftilde$. However, the spliced filling $\ftilde$ could potentially have an additional quinv triple, formed by the cells with entries $\alpha_2, \beta_2, \gamma$. But this is precluded by the inequality $\gamma \geq \beta_2$ above.

\medskip
\noindent
{\bf Case 4:} $c$ is in the $(j+1)^{th}$ column and $c = \ctilde$. If $(x,y,c)$ is a quinv-triple of $\fdag$ and $(x,y) = (\ctilde[x],\ctilde[y])$, then $(x,y,c)$ is a quinv-triple of $\ftilde$. It remains only to consider the case when  $(x,y,c)$ with $x,y$ in the $j^{th}$ column and $y \neq \ctilde[y]$ is a quinv-triple of $\fdag$. In this case, we have from \eqref{def:mdef}, \eqref{eq:splicecor}  that $\gamma = \alpha_2 < \beta_2 < \beta_1$ and $\alpha_1 < \beta_2$. Given these, we obtain that the cells with entries $(\alpha_1, \beta_1, \alpha_2)$ in $\fdag$ and $(\alpha_1, \beta_2, \alpha_2)$ in $\ftilde$ are both quinv-triples or are both non quinv-triples, depending on whether $\alpha_1 < \alpha_2$ or not.

\medskip
\noindent
{\bf Case 5:} $c$ is in the $(j+1)^{th}$ column and $c \neq \ctilde$. We have $\gamma = \alpha_2 < \beta_2 \leq \alpha_1$ (from \eqref{eq:spliceeq}). Thus the cells with entries $\alpha_1, \beta_1, \alpha_2$ of $\fdag$ do not form a quinv-triple. Thus in any quinv-triple of $\fdag$ of the form $(x,y,c)$, $x$ occurs strictly to the left of column $j$. These bijectively correspond to quinv-triples $(x,y,\ctilde)$ in $\ftilde$. 

\medskip
\noindent
This completes the proof.
\end{proof}

It follows from Lemma~\ref{lem:keylem} that for $j<n-1$ and $c \in \cells(i,j,\fdag)$, we have $\zcount(c, \fdag) = \zcount(\ctilde, \ftilde)$, where $\ftilde$ is the result of an elementary splice operation applied to $\fdag$. Iterating, we conclude that $\zcount(c, \fdag) = \zcount(c', \dsplice(F))$ where $c'$ is the cell corresponding to $c$ at the end of the succession of elementary splice moves. Using \eqref{eq:cfcfdag}, we conclude that for $j < n-1$, the association $c \mapsto c'$ is a bijection from $\cells(i,j,F) = \cells(i,j,\fdag)$ to $\cells(i,j,\dsplice(F))$. Further \eqref{eq:zcffdag} implies that  
\[ \zcount(c, F) = \zcount(c', \dsplice(F)) \]
for all $c \in \cells(i,j,F)$ with $j<n-1$. 
Thus the overlays $\Lambda_{ij}$ from \eqref{eq:lamij} remain unchanged in $F$ and $\dsplice(F)$ for $j < n-1$. This completes the proof that the diagram (2B) of Theorem~\ref{thm:mainthm} is indeed commutative for $\psi_v = \psi_{\quinv}$. We use \eqref{eq:boxbr} and part (3) of Theorem~\ref{thm:mainthm} to obtain the commutativity for $\psi_v = \psi_{\inv}$. \qed

\subsection{}
Finally, the sole assertion of Theorem~\ref{thm:mainthm} that remains to be proved is that $\psi_{\inv}$ and $\psi_{\quinv}$ are bijections. We describe the construction of $\psi_{\inv}^{-1}$. Given a partition $\lambda$ and $(T, \Lambda) \in \pop$ with $\Lambda = (\Lambda_{ij})$, construct the filling $F:=\psi_{\inv}^{-1}(T,\Lambda) \in \csf$ inductively row-by-row, from the bottom ($n^{th}$) row to the top as follows:

(a) fill all cells (if any) of the $n^{th}$ row  with the entry $n$.

(b) let $1 \leq i < n$; assuming that all rows of $F$ strictly below row $i$ have been completely determined, we will now determine the entries of row $i$. For each $i \leq j <n$, we must have $\NE_{ij}(T)$ many cells in row $i$ containing the entry $j+1$. We will determine the locations of the entries $j+1$ inductively, starting with $j= n-1$ and going down to $j= i$. Once done, we will fill the remaining cells in row $i$ with the entry $i$ and move up to row $i-1$. 

Fix $i \leq j <n$ and assume that the locations of all entries $j'+1$ for $j<j'<n$ in row $i$ of $F$ have been determined (vacuously if $j=n-1$). Call this partial filling $F_{ij}$.

The number of empty cells in the $i^{th}$ row of $F_{ij}$ is $T^{j+1}_i$, since we require $\rsort(F) = T$. Among these, it is clear that exactly $T^{j+1}_{i+1}$ of them will have a cell directly below containing an entry $\leq (j+1)$. Thus, the number of candidate cells in row $i$ in which we can potentially put a $j+1$ without violating the column strictness condition thus far is exactly $T^{j+1}_i - T^{j+1}_{i+1} = \NE_{ij}(T) + \SE_{ij}(T)$. 

Let $\cand(i,j)$ denote this collection of cells, i.e., $\cand(i,j)$ comprises cells $c$ in the $i^{th}$ row of $\dg(\lambda)$ such that:
\begin{enumerate}
\item[(i)] $c$ is empty in $F_{ij}$, and \item[(ii)] if there is a cell $d$ immediately below $c$ in $\dg(\lambda)$, then $F_{ij}(d) > (j+1)$. 
\end{enumerate}
Let $k = \NE_{ij}(T)$ and $\ell = \SE_{ij}(T)$. We label the  cells of $\cand(i,j)$ from right-to-left\footnote{This labelling should be done left-to-right when defining $\psi_{\quinv}^{-1}$. This is the only change required. The rest of the proof works with minimal changes.} as $0, 1, \cdots, (k+\ell-1)$. We now use the identification from equation~\eqref{eq:gamma-a} of partitions fitting inside a $(k \times \ell)$-box with strictly decreasing $k$-tuples of integers in $0, 1, \cdots, k+\ell-1$. Via this, the (given) partition $\Lambda_{ij}$ can be identified with a decreasing tuple $a_1 > a_2 > \cdots > a_k$. Consider the cells in $\cand(i,j)$ with the labels $(a_p)_{p=1}^k$ and put the entry $j+1$ into these. We proceed downward by induction until we reach $j=i$.

\noindent
(c) Finally, fill the remaining cells of row $i$ with the entry $i$. Proceed by downward induction on $i$ until all entries of $F$ are determined.

\medskip
We now show that $\psi_{\inv}(F) = (T, \Lambda)$. It is clear by the construction of $F$ above that $\rsort(F) = T$, since at each step of the process, we filled exactly $\NE_{ij}(T)$ cells in row $i$ with the entry $j+1$. It remains to show that the overlay corresponding to $F$ is exactly $\Lambda$. We have:
\begin{lemma}\label{lem:reflem} Let $1 \leq i \leq j < n$ and $c \in \cells(i,j,F)$. Then $(x,y,z)$  is a refinv-triple for $F$ with $z=c$ if and only if  $x \in \cand(i,j) - \cells(i,j,F)$, $x$ occurs to the right of $z$ and $y=\down[x]$.
\end{lemma}
\begin{proof} To prove the forward implication: since $F(z) = j+1$, we have (i) $F(x) \leq j$ and (ii) $F(y) > (j+1)$. By (i), the cell $x$ must have been empty at the stage when the candidates $\cand(i,j)$ were being picked. By (ii), we conclude that $x$ must indeed have been one of the candidate cells. That it occurs to the right of $z$ and $y=\down[x]$ follow from the definition of refinv-triples.  The converse follows similarly.
\end{proof}
We use the foregoing lemma to complete the proof. From the construction of $F$, we know that $\cells(i,j,F)$ are precisely the cells in $\cand(i,j)$ with labels $(a_p)_{p=1}^k$. Let $c$ denote the cell with label $a_p$. Since there are $a_p$ (resp. $k-p$) cells to its right in $\cand(i,j)$ (resp. in $\cells(i,j,F)$), we obtain from Lemma~\ref{lem:reflem} that
\[ \zcb(c, F) = a_p - (k-p) \]
Equation~\eqref{eq:gamma-a} completes the proof. \qed

\medskip
We remark that the construction of $\psi_{\quinv}^{-1}$ is entirely analogous. As mentioned already in the footnote, the  point of departure is the numbering of the cells of $\cand(i,j)$ - these are numbered $0, 1, 2, \cdots$ from left-to-right in this case. 
\subsection{}\label{sec:eg}
For example, let $n=4$, $\lambda = (10, 6, 4,0)$ and let $T, \Lambda$ be the GT pattern and overlay depicted in Figure~\ref{fig:gt-ssyt}. 
%% \[
%% T =\begin{ytableau}
%%   1  & 1 & 1 &1 &2 &2&2&3&4&4 \\
%%   2 &2&3&3 &3&4\\
%%   3 &3&4&4
%% \end{ytableau}\]
Then $\psi_{\quinv}^{-1}(T,\Lambda)$ is precisely the CSF $F$ of Figure~\ref{fig:zcfig}, i.e.,
\[\ytableausetup{mathmode, smalltableaux}
\psi_{\quinv}^{-1}(T,\Lambda) = \ytableaushort{1121212443,223334,3344}\]
while
\[\ytableausetup{mathmode, smalltableaux}
\psi_{\inv}^{-1}(T,\Lambda) = \ytableaushort{2111321442,332243,4433}\]

\subsection{CL bases: Proof of Proposition~\ref{prop:clbasis-csf-Et}}
Observe that equations~\eqref{eq:qwtquinv}, \eqref{eq:qwtinv} prove part (1) of Proposition~\ref{prop:clbasis-csf-Et} while part (2) is implied by Proposition~\ref{prop:zcandzcb}. Part (3) of Proposition~\ref{prop:clbasis-csf-Et} follows from the definition of the Chari-Loktev basis in equations~\eqref{eq:lij1}-\eqref{eq:clbasis-mon}, and from the definitions of the partition overlays in equations \eqref{eq:lamij} and \eqref{eq:lamijbar}. \qed

\begin{figure}
  \begin{center}
  \ytableausetup{mathmode, smalltableaux}
$F=\;$\begin{ytableau}
  *(cyan) 1  & *(yellow) 2 & *(cyan) 1 & *(yellow) 2 \\
 *(pink) 3 & *(green) 4  
  \end{ytableau} \qquad \quad
 $\zcount(\cdot,F)=\;$ \begin{ytableau}
  *(cyan) 0  & *(yellow) 1 & *(cyan) 0 & *(yellow) 2 \\
 *(pink) 0 & *(green) 1  
  \end{ytableau} \qquad \quad
$\zcb(\cdot,F)=\;$ \begin{ytableau}
  *(cyan) 0  & *(yellow) 1 & *(cyan) 0 & *(yellow) 0 \\
 *(pink) 0 & *(green) 0  
  \end{ytableau}
 \caption{Here $F \in \csf$ for $\lambda=(4,2,0,0)$ and $n=4$. The cellwise zcount and $\zcb$ values are also indicated.}
  \label{fig:zc-zcb-fig}
\end{center}
\end{figure}

\smallskip
\begin{example}
We exhibit the Chari-Loktev monomials of equations~\eqref{eq:bquinv} and \eqref{eq:binv} in some examples below. If
\ytableausetup{mathmode, smalltableaux}
$F=\;$\begin{ytableau}
  *(cyan) 1  & *(yellow) 2 & *(cyan) 1 & *(yellow) 2 \\
 *(pink) 3 & *(green) 4  
  \end{ytableau}, then
 $\zcount(\cdot,F)=\;$ \begin{ytableau}
  *(cyan) 0  & *(yellow) 1 & *(cyan) 0 & *(yellow) 2 \\
 *(pink) 0 & *(green) 1  
  \end{ytableau} and so 
  \[\mathfrak{b}_{\quinv}(F) = (E_{21}\otimes t^2)(E_{21}\otimes t)(E_{42}\otimes t).\]
Likewise, in this case, we have $\zcb(\cdot,F)=\;$ \begin{ytableau}
  *(cyan) 0  & *(yellow) 1 & *(cyan) 0 & *(yellow) 0 \\
 *(pink) 0 & *(green) 0  
  \end{ytableau} and
$\mathfrak{b}_{\inv}(F) = E_{21}\otimes t$.

\smallskip
\noindent
A larger example is afforded by the CSF in Figure~\ref{fig:zcfig}. For that $F$, we have:
\[ \mathfrak{b}_{\quinv}(F) = (E_{21}\otimes t^2)\,(E_{21}\otimes t)\,(E_{31}\otimes t^2) \,(E_{41}\otimes t)^2\,(E_{42}\otimes t)\,(E_{43}\otimes t^2)^2. \]
\end{example}

\section{Local Weyl modules and limit constructions}\label{sec:dirlim}
\subsection{} In this section, we show how the bijections of Theorem~\ref{thm:mainthm} may be applied to the study of the basic representation $L(\Lambda_0)$ (level 1 vacuum module) of the affine Kac-Moody algebra $\widehat{\mathfrak{sl}_n}$ \cite{kac}. Our main result is Proposition~\ref{prop:charcsf}, which gives an elegant  expression for the character of $L(\Lambda_0)$ in terms of column strict fillings, which seems not to have been noticed before. 

The corresponding theory in terms of partition overlaid patterns  appears in \cite{RRV-CLpop}. The thrust of the arguments below lies in translating this to the CSF formalism. We refer the reader to \cite{RRV-CLpop} for any unspecified details about the POP picture.

\subsection{} As mentioned earlier in Section~\ref{sec:locweyl}, the $q$-Whittaker polynomial $\wl$ is the graded character of the local Weyl module - this follows by comparing the monomial expansion \eqref{eq:ferm} with \cite[Prop. 2.1.4]{ChariLoktev-original}, or more transparently, by comparing \eqref{eq:fermpop} with \cite[Equation (54)]{RRV-CLpop}. 

Since local Weyl modules for the current algebra $\mathfrak{sl}_n[t]$ coincide with certain Demazure modules (the $\mathfrak{sl}_n$-invariant ones) of $L(\Lambda_0)$ \cite{ChariLoktev-original,fourier-littelmann}, one obtains the character of $L(\Lambda_0)$ as a limit of the characters of local Weyl modules. 

To formulate this precisely, we must first renormalize the $q$-Whittaker polynomial $\wl$. Given $\gamma = (c_1, c_2, \ldots, c_n) \in \mathbb{R}^n$, we view it as an element of the Cartan subalgebra of $\mathfrak{gl}_n$. Let $|\gamma| = \sum_{i=1}^n c_i$ and
\[ \lVert \gamma \rVert^2 := \sum_{i=1}^n c_i^2 - \frac{|\gamma|^2}{n}.\]
This is just the square norm of the projection\footnote{The projection being given by $\gamma - \frac{|\gamma|}{n} (1,1,\cdots,1)$.} of $\gamma$ to the Cartan subalgebra of $\mathfrak{sl}_n$. Given a partition $\lambda$ of length $<n$ with $|\lambda|$ divisible by $n$, the renormalized \qw polynomial is defined to be:
\[ \wlhat := q^{\lVert \lambda \rVert^2/2} \,\left(x_1x_2\cdots x_n\right)^{-|\lambda|/n}\; \wl[q^{-1}] \]
It is easy to see that $\wlhat$ is a symmetric, homogeneous Laurent polynomial of degree zero in the $x_i$, with $\mathbb{Z}_+[q]$-coefficients. Further, its constant term (independent of the $x_i$) is $1$.

We now state the following limit theorem:
\begin{proposition}\label{prop:limprop}
    Let $n \geq 2$. Let $\lambda$ be a partition of length $<n$, with $|\lambda|$ a multiple of $n$. Let $\theta=(2,1,1,\cdots,1,0)$ be the highest root of $\mathfrak{sl}_n$, i.e., the partition with $n-1$ nonzero parts and $|\theta|=n$. Then:
\begin{equation}\label{eq:eqlimprop}
\lim_{k \to \infty} \wlhat[\lambda+k\theta] = \frac{\Theta(X_n,q)}{\prod_{k\geq 1} (1-q^k)^{n-1}}
\end{equation}
where $\Theta(X_n, q)$ is the theta-function of the root lattice of $\mathfrak{sl}_n$:
\[ \Theta(X_n, q) = \sum_{\gamma \in \mathbb{Z}^n,\, |\gamma|=0} q^\frac{\lVert \gamma \rVert^2}{2} x^\gamma\]
\end{proposition}
We recall from \cite[Chap. 12]{kac} that the right hand side is precisely the (normalized) character of the basic representation $L(\Lambda_0)$ of $\widehat{\mathfrak{sl}_n}$. We let
\[ \chil :=\frac{\Theta(X_n,q)}{\prod_{k\geq 1} (1-q^k)^{n-1}}.\]
\begin{proof}
   We recall from \cite{fourier-littelmann}, \cite[\S 6]{RRV-CLpop} that we have a chain of injections of $\mathfrak{sl}_n[t]$-modules:
\begin{equation}\label{eq:wloc-chain}
W_{\mathrm{loc}}(\lambda) \hookrightarrow W_{\mathrm{loc}}(\lambda+\theta) \hookrightarrow W_{\mathrm{loc}}(\lambda+2\theta) \hookrightarrow \cdots
\end{equation}
and that the direct limit of these modules is $L(\Lambda_0)$. The highest weight vector $v_{\Lambda_0}$ of $L(\Lambda_0)$ occurs as the unique (up to scaling) vector of highest grade in each $W_{\mathrm{loc}}(\lambda+k\theta)$; the renormalization $\wlhat[\lambda+k\theta]$ re-indexes the grading so as to make this a vector of lowest grade (with grade zero). The assertion now follows from the fact that $W_{\lambda+k\theta}(X_n; q)$ is the character of $W_{\mathrm{loc}}(\lambda+k\theta)$ for all $k$.
 \end{proof}

We refer the reader to \cite[\S 4]{fourier-littelmann} and \cite{FL-tplimit} for a more detailed discussion of the limit construction of Proposition~\ref{prop:limprop} in terms of Demazure modules, and \cite{magyar}, \cite{kuniba-et-al} for the ``crystals" point-of-view. Finally, an analogous formula to \eqref{eq:eqlimprop} involving the limit of modified Hall-Littlewood polynomials may be found in \cite[\S 1.3]{Kirillov_newformula}.

\subsection{} 
As shown in \cite[\S 6.1.1]{RRV-CLpop}, the chain of modules in equation~\eqref{eq:wloc-chain} is mirrored by a corresponding chain of inclusions between sets of partition overlaid patterns . The essential property of these inclusions that is relevant to our present context is summarized in the proposition below - the proof follows directly from the arguments of \cite[\S 5]{RRV-CLpop}. We refer the reader to \cite[\S 6]{RRV-CLpop} for the definition of the map $\mathcal{S}$ of the following proposition.
\begin{proposition}\label{prop:seqprops} Under the hypotheses of Proposition~\ref{prop:limprop}, 
    there is a sequence of injective maps:
    \begin{equation}\label{eq:popseq}
    \pop \stackrel{\mathcal{S}}{\to} \pop[\lambda+\theta] \stackrel{\mathcal{S}}{\to} \pop[\lambda+2\theta] \stackrel{\mathcal{S}}{\to} \cdots
    \end{equation}
    with the following properties:
    \begin{enumerate}
    \item Given $k \geq 0$ and $(T, \Lambda) \in \pop[\lambda+k\theta]$, define \[d(T,\Lambda) := \frac{\lVert \lambda+k\theta \rVert^2}{2} - |\Lambda|.\] Then $d(T,\Lambda) \geq 0$. 
    \smallskip\item If $(T',\Lambda') = \mathcal{S}(T,\Lambda)$, then $d(T,\Lambda) = d(T', \Lambda')$ and $\xtil^{T'} = \xtil^T$ where 
    \[ \xtil^F := x^F \left(x_1x_2\cdots x_n\right)^{-|\mu|/n} \]
    is the renormalized monomial weight of $F \in \csf[\mu]$ for any partition $\mu$ with $|\mu|$ a multiple of $n$. \qed
    \end{enumerate}
\end{proposition}

\noindent
We next consider the direct limit of \eqref{eq:popseq}. Let 
$\mathcal{P} := \sqcup_{k \geq 0}\, \pop[\lambda + k\theta]$. Define the equivalence relation $\sim$ on $\mathcal{P}$ as the symmetric, transitive closure of the relation   
\begin{equation}\label{eq:popreln}
    (T,\Lambda) \sim \mathcal{S}(T,\Lambda)
\end{equation} 
for all  $(T,\Lambda) \in \mathcal{P}$. It is clear from Proposition~\ref{prop:seqprops} that $\xtil^T$ and $d(T,\Lambda)$ only depend on the equivalence class of $(T,\Lambda)$.
We may view $\mathcal{P}/\!\!\sim$ as the direct limit of \eqref{eq:popseq}:
\[ \mathcal{P}/\!\!\sim \;= \lim_{\stackrel{\longrightarrow}{k}} \pop[\lambda+k\theta]. \]
Finally, we obtain from the above discussion and \cite[\S 5]{RRV-CLpop} the following combinatorial shadow of the character identity \eqref{eq:eqlimprop}:
\begin{equation}\label{eq:charpops}
 \chil = \sum_{(T,\Lambda) \in \mathcal{P}/{\sim}} \xtil^T q^{d(T,\Lambda)} 
\end{equation}
where the sum ranges over distinct representatives of the equivalence classes of $\sim$.

\subsection{}
We will now use Theorem~\ref{thm:mainthm} to replace POP with CSF in the arguments of the preceding subsection. 
Consider the sequence of injective maps:
\begin{equation}\label{eq:csfdl}
\csf[\lambda] \stackrel{s}{\to} \csf[\lambda+\theta] \stackrel{s}{\to} \csf[\lambda+2\theta] \stackrel{s}{\to}  \cdots 
\end{equation}
where for $F \in \csf[\lambda+k\theta]$, we define $s(F) \in \csf[\lambda+(k+1)\theta]$ to be the filling obtained by attaching to $F$ a leftmost column containing \[\ytableausetup{mathmode, smalltableaux} \ytableaushort{2,3,\cdot,\vdots,n}\] and a rightmost column containing $\ytableaushort{1}$. For example, if $n=4$ and $F =\ytableaushort{12,3}$, then $s(F) = \ytableaushort{{*(cyan) 2}12{*(cyan) 1},{*(cyan) 3}3,{*(cyan) 4}}$.

The following key proposition relates the two injections $\mathcal{S}$ (among partition overlaid patterns ) and $s$ (among column strict fillings):
\begin{proposition}\label{prop:cpcomm}
With notation as above, the following diagram commutes for all $k \geq 0$:
    \[\begin{tikzcd}
	\csf[\lambda+k\theta] &&& \csf[\lambda+(k+1)\theta] \\
	\\
	\pop[\lambda+k\theta] &&& \pop[\lambda+(k+1)\theta] 
	\arrow["{s}", from=1-1, to=1-4]
	\arrow["\psi_{\inv}", from=1-1, to=3-1]
	\arrow["{\mathcal{S}}"', from=3-1, to=3-4]
	\arrow["\psi_{\inv}", from=1-4, to=3-4]
\end{tikzcd}\]
\end{proposition}
This proposition is a simple consequence of the definitions of the four maps involved in the commutative diagram. \qed

 An analogous statement can be made with $\quinv$ as well, but the map $\psi_{\quinv}^{-1} \,\mathcal{S} \,\psi_{\quinv}$ does not have as simple a description as the map `$s$' above in the case of $\inv$.

\subsection{}
Let $k \geq 1$ and define $\mathcal{C}_k(\lambda)$ to be the set of  $F \in \csf[\lambda+k\theta]$ such that the entry $\boxed{1}$ either occurs in the first column of $F$ or does not occur in its last column. We also set $\mathcal{C}_0(\lambda) :=\csf[\lambda]$ and \[ \mathcal{C} := \bigsqcup_{k \geq 0} \mathcal{C}_k(\lambda).\] 

By analogy to \eqref{eq:popreln}, we define an equivalence relation on the set \[\widehat{C} = \bigsqcup_{k \geq 0} \csf[\lambda+k\theta]\] by taking the symmetric, transitive closure of the relations $F \sim s(F)$ for all $F \in \widehat{C}$. 
It is clear from the definition of the map `$s$' that the elements of $\mathcal{C}$ form a set of distinct representatives of the equivalence classes of $\widehat{C}/\!\!\sim$.

Now, in conjunction with part 1 of Theorem~\ref{thm:mainthm} and equation~\eqref{eq:charpops}, Proposition~\ref{prop:cpcomm} implies the following result:

\begin{proposition}\label{prop:charcsf} With notation as above, we have
  \[  \chil = \sum_{k \geq 0} \sum_{F \in \mathcal{C}_k(\lambda)} \,\xtil^F \,q^{\frac{\lVert \lambda+k\theta \rVert^2}{2}-\inv(F)} \]   
\end{proposition}

In particular, taking $\lambda$ to be the empty partition, we obtain:
%\cite[Corollary 5.13]{RRV-CLpop}, we deduce \cite{brv-fullversion}:
\begin{corollary}\label{cor:kthetal}
  Fix $n \geq 2$ and consider the partition $\theta=(2,1,1,\cdots,1,0)$ with $n-1$ nonzero parts and $|\theta|=n$. For $k \geq 0$, let $\mathcal{C}_k$ denote the set of CSFs $\,F$ of shape $k\theta$ and entries in $[n]$, with the property that either $1$ occurs in the first column of $F$ or $1$ does not occur in its last column. Then
    \[\chil = \sum_{k \geq 0} \sum_{F \in \mathcal{C}_k} \,\xtil^F \,q^{k^2-\inv(F)}.\]
\end{corollary}

\section{Lattice path formalism}\label{sec:lattpathwheeler}
\subsection{} This section is inspired by \cite{wheeler-lectures} which contains an interpretation of the \qw polynomial in terms of coloured paths on a square lattice (see also \cite{borodin-wheeler-spinwhittaker} and \cite{GarbaliWheeler, borodin-wheeler}). Our lattice paths are modelled on those of \cite[Talk  3]{wheeler-lectures}, with a mild change of convention in the way the paths are drawn for better conformity with our formalism thus far. As in \cite{wheeler-lectures}, the $q$-weight of an ensemble of paths will be the total number of intersections of a certain type. The new insight stemming from our Theorem~\ref{thm:mainthm} is that rather than just the gross total, keeping track of intersections in {\em each tile} of the square lattice encodes the partition overlay data for {\em quinv}. Moreover, keeping track of non-intersections provides a simultaneous readout of the {\em inv} overlay as well. Our lattice path diagrams thus provide a visual demonstration\footnote{There is yet another lattice path model for the \qw polynomial in which the paths are allowed to ``wrap around" the grid, with the $q$-weights being related to the location of the wraparounds \cite[Talk 3]{wheeler-lectures}, \cite{GarbaliWheeler}. A natural bijection between these models is not known.} of Theorem~\ref{thm:mainthm}.

Following \cite{GarbaliWheeler, wheeler-lectures}, we start with a grid of cells with $n$ rows and $n$ columns, with the labelling shown in Figure~\ref{fig:lat-sqr}. Each cell of the grid will be called a {\em tile}. Our labelling conventions differ slightly from \cite{GarbaliWheeler, wheeler-lectures}.

\begin{figure}[h]   
\begin{center}
\vspace{0.3cm}
\begin{tikzpicture}[scale=0.5,baseline=(current bounding box.center)]
\foreach\x in {0,...,4}{
\draw (0,2*\x) -- (4,2*\x);
}
\foreach\x in {0,...,4}{
\draw[dashed] (4,2*\x) -- (6,2*\x);
}
\foreach\x in {0,...,4}{
\draw (6,2*\x) -- (8,2*\x);
}
\foreach\y in {0,...,4}{
\draw (2*\y,0) -- (2*\y,2);
}
\foreach\y in {0,...,4}{
\draw[dashed] (2*\y,2) -- (2*\y,4);
}
\foreach\y in {0,...,4}{
\draw (2*\y,4) -- (2*\y,8);
}
\node[left] at (-0.5,7) {\tiny $1$};
\node[left] at (-0.5,5) {\tiny$2$};
\node[left] at (-0.5,1) {\tiny $n$};
\node[left] at (7.5,8.5) {\tiny $1$};
\node[left] at (1.5,8.5) {\tiny$n$};
\end{tikzpicture}~
\vspace{0.3cm}
\end{center}
\caption{}
\label{fig:lat-sqr}
\end{figure}

Now let $F \in \csf$. If $F$ is just a single column, we represent it as a lattice path as follows: if the entries of $F$ are $i_1 <i_2 < \cdots < i_k$, its lattice path: (i) starts at the right edge of the grid in row $i_1$ and moves left or down (ii) has exactly $k$ horizontal steps, these being in rows $i_1, i_2, \cdots, i_k$, and  (iii) the horizontal steps (except the first) straddle two successive columns of the grid. For example, 
\ytableausetup{mathmode, smalltableaux}
$F=\;\begin{ytableau}
  1 \\
   3  \\
   4  
  \end{ytableau}\;$ is represented by the lattice path shown in Figure~\ref{fig:quinv_dig_singlecol}.
 
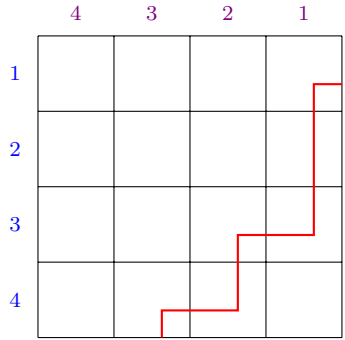
\begin{figure}[ht]
\begin{center}
\begin{tikzpicture}    
\draw(1,2)node{\begin{tikzpicture} \draw[] grid (4,4); \end{tikzpicture}};
\draw[color=red,thick](3,3.36)--(2.63,3.36)--(2.63,1.36)--(1.63,1.36)--(1.63,0.36)--(0.63,0.36)--(0.63,0.0);
\foreach\i in{0,...,3}{
\pgfmathtruncatemacro{\k}{4-\i};
\draw[color=blue](-1.3,\i+0.5)node{$\scriptstyle \k$};};
\foreach\j in{0,...,3}{
\pgfmathtruncatemacro{\l}{4-\j};
\draw[color=violet] (\j-0.5,4.3)node {$\scriptstyle \l$};};
\end{tikzpicture}
\caption{Lattice path corresponding to a single-column CSF.}
\label{fig:quinv_dig_singlecol}
\end{center}
\end{figure}

 Given a lattice path $\beta$ and a tile $X$, let $\beta_X$ denote the (possibly empty) intersection of $\beta$ and $X$, i.e., the portion of $\beta$ that lies in $X$. We observe that a lattice path can intersect a tile in only one of 3 possible configurations - these are labelled I, II, III in Figure~\ref{fig:csf-dig}. 

\begin{figure}[h]
\begin{center}
\begin{tikzpicture}
\draw(0.0,0.0)node{
\begin{tikzpicture}[scale=1]
\draw[] grid (1,1);
\draw(0.54,1.00)--(0.54,0.00);
\draw(0.50,-0.30)node{I};
\end{tikzpicture}};
\draw(4.0,0.0)node{\begin{tikzpicture}[scale=1]
\draw[] grid (1,1);
\draw(1.00,0.45)--(0.54,0.45)--(0.54,0.00);
\draw(0.50,-0.30)node{II};
\end{tikzpicture}};
\draw(8.0,0.0)node{
\begin{tikzpicture}[scale=1]
\draw[] grid (1,1);
\draw(0.54,1.00)--(0.54,0.45)--(0.00,0.45);
\draw(0.50,-0.30)node{III};

\end{tikzpicture}
};

\end{tikzpicture}
\caption{A lattice path passing through a tile: possible configurations.}
\label{fig:csf-dig}
\end{center}
\end{figure}
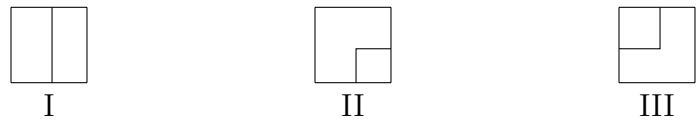

If $F$ has multiple columns, each column is individually represented as a lattice path as above, subject to the following {precedence rule} which determines the relative positioning of the paths. Let $C, C^{\,\prime}$ be columns of $F$ and let $\gamma, \gamma'$ denote the lattice paths they correspond to.
If the column $C$ occurs to the {\em left} of $C^{\,\prime}$ in $F$, then the precedence rule for drawing the paths states that for every tile $X$:
\begin{enumerate}
    \item The horizontal segment of $\gamma_X$ occurs {\em below} the horizontal segment of $\gamma'_X$ when both such segments exist, i.e., when $\gamma_X$ and $\gamma'_X$ are both nonempty and not of type I.
    \item The vertical segment of $\gamma_X$ occurs {\em to the right} of the vertical segment of $\gamma'_X$ when both segments exist, i.e., when $\gamma_X$ and $\gamma'_X$ are both nonempty. 
    \end{enumerate}
Informally, we may summarize this as ``$\gamma$ stays below and to the right of  $\gamma'$ wherever possible". Figure~\ref{fig:quinv_dig} shows an example of a CSF $F$ and its associated lattice path diagram, with the colour coding of the columns of $F$ matching those of the corresponding lattice paths. Note that this is our running example - this is the same $F$ considered in Figure~\ref{fig:zcfig} earlier.

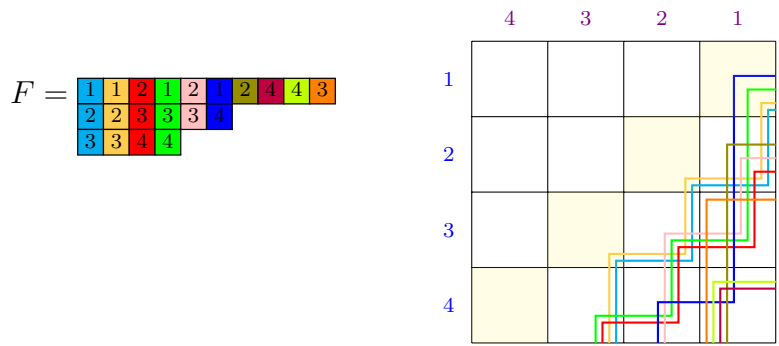
\begin{figure}[ht]
\begin{center}
\begin{tikzpicture}
%\begin{scope}
\draw(-5,3)node{
\ytableausetup{mathmode, smalltableaux}
$F=\;$\begin{ytableau}
  *(cyan) 1  & *(Mycolor8) 1 & *(red) 2 & *(green) 1 & *(pink) 2 & *(blue) 1 & *(olive) 2 & *(purple) 4 & *(lime) 4 & *(orange) 3\\
   *(cyan) 2  & *(Mycolor8) 2 & *(red) 3 & *(green) 3 & *(pink) 3 & *(blue) 4 \\
  *(cyan) 3  &*(Mycolor8) 3 & *(red) 4 & *(green) 4  
  \end{ytableau}};
%\end{tikzpicture}
%\begin{tikzpicture}[scale=1]
\draw(1,2)node{\begin{tikzpicture} \draw[] grid (4,4); \end{tikzpicture}};
\draw[color=cyan,thick](3,3.09)--(2.90,3.09)--(2.90,2.09)--(1.90,2.09)--(1.90,1.09)--(0.90,1.09)--(0.90,0.0);
\draw[color=Mycolor8,thick](3,3.18)--(2.81,3.18)--(2.81,2.18)--(1.81,2.18)--(1.81,1.18)--(0.81,1.18)--(0.81,0.0);
\draw[color=red,thick](3,2.27)--(2.72,2.27)--(2.72,1.27)--(1.72,1.27)--(1.72,0.27)--(0.72,0.27)--(0.72,0.0);
\draw[color=green,thick](3,3.36)--(2.63,3.36)--(2.63,1.36)--(1.63,1.36)--(1.63,0.36)--(0.63,0.36)--(0.63,0.0);
\draw[color=pink,thick](3,2.45)--(2.54,2.45)--(2.54,1.45)--(1.54,1.45)--(1.54,0.00);
\draw[color=blue,thick](3,3.54)--(2.45,3.54)--(2.45,0.54)--(1.45,0.54)--(1.45,0.00);
\draw[color=olive,thick](3,2.63)--(2.36,2.63)--(2.36,0.00);
\draw[color=purple,thick](3,0.72)--(2.27,0.72)--(2.27,0.00);
\draw[color=lime,thick](3,0.81)--(2.18,0.81)--(2.18,0.00);\draw[color=orange,thick](3,1.90)--(2.09,1.90)--(2.09,0.00);
\fill [yellow, opacity=0.1] (-1,0) rectangle (0,1);
\fill [yellow, opacity=0.1] (0,1) rectangle (1,2);
\fill [yellow, opacity=0.1] (1,2) rectangle (2,3);
\fill [yellow, opacity=0.1] (2,3) rectangle (3,4);
\foreach\i in{0,...,3}{
\pgfmathtruncatemacro{\k}{4-\i};
\draw[color=blue](-1.3,\i+0.5)node{$\scriptstyle \k$};};
\foreach\j in{0,...,3}{
\pgfmathtruncatemacro{\l}{4-\j};
\draw[color=violet] (\j-0.5,4.3)node {$\scriptstyle \l$};};
%\end{scope}
\end{tikzpicture}
\caption{A CSF and its lattice path diagram (colour coded).}
\label{fig:quinv_dig}
\end{center}
\end{figure}

Consider the tiles on the antidiagonal of the grid (shaded yellow in the example of Figure~\ref{fig:quinv_dig}). Let $X$ be a tile to the right of and/or below this antidiagonal, i.e., a tile with column label $i$ and row label $j+1$, with $1 \leq i \leq j <n$; we denote this tile $X_{ij}$. We have the following simple lemma, which follows from equations~\eqref{eq:cellcard} and \eqref{eq:cupc}.
\begin{lemma}\label{lem:nesepic}
Let $1 \leq i \leq j <n$. With notation as above, the number of type II lattice paths in $X_{ij}$ is $\NE_{ij}(T)$ and the number of type I paths is $\SE_{ij}(T)$. \qed 
\end{lemma}

\subsection{}
Given $F \in \csf$, we can now extract our bijections $\psi_{\inv}$ and $\psi_{\quinv}$ from the lattice paths diagram of $F$ via the following steps:
\begin{enumerate}
    \item {\em Declutter:} If a lattice path $\gamma$ in the diagram intersects a tile $X$ in configuration III, then delete that portion of $\gamma$, i.e., delete $\gamma_X$ from the diagram (it will be irrelevant to what follows), leaving the rest of $\gamma$ unchanged.
    \smallskip\item {\em Count intersections:} Lattice paths in each tile are now in either of the configurations I or II. In each tile, put a solid circle  to mark the points where a path of type I intersects one of type II:  
    
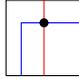
\[\begin{tikzpicture}[scale=1]
\draw[] grid (1,1);
\draw[color=red](0.5,1.00)--(0.5,0.00);
\draw[color=blue](1.00,0.70)--(0.20,0.70)--(0.20,0.00);
\fill (0.5,0.7) circle[radius=0.6mm];
\end{tikzpicture}\]
It readily follows from the definition of quinv-triples and our path drawing convention that each such intersection corresponds to a quinv triple in $F$. 
    \item {\em Count non-intersections:} Likewise, when a tile contains a path of type I and one of type II which fail to intersect, we keep track of this non-intersection by drawing an open circle at the point of the type I path where the extended horizontal segment of the type II path would have intersected it:
    
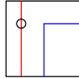
\[\begin{tikzpicture}[scale=1]
\draw[] grid (1,1);
\draw[color=red](0.2,1.00)--(0.2,0.00);
\draw[color=blue](1.00,0.70)--(0.50,0.70)--(0.50,0.00);
\draw (0.2,0.7) circle[radius=0.6mm];
\end{tikzpicture}\]
As above, each such non-intersection corresponds to an inv triple in $F$. 
    \item As in Lemma~\ref{lem:nesepic}, consider the tiles $X_{ij}, \, 1 \leq i \leq j <n$ that are strictly to the right and below the antidiagonal of the grid. The solid circles in each of these tiles form the Ferrers diagram of a partition (after a left-right flip). 
    For example, in Figure~\ref{fig:quinv_dig_mod}, the tile $X_{11}$ (the last tile in the second row) has the solid dot configuration \[\begin{matrix} {\bullet} & {\bullet} \\ & {\bullet} \end{matrix}\] which we interpret as the partition $2+1$. These partitions encode the overlay for $\psi_{\quinv}$; more precisely, let $\Lambda_{ij}$ be the partition formed by the solid circles in the tile $X_{ij}$ and let $\Lambda=(\Lambda_{ij})_{1 \leq i \leq j <n}$ be the tuple. Then $\psi_{\quinv}(F) = (\rsort(F),\Lambda)$. In particular, $\quinv(F)$ is the total number of solid circles in the lattice paths diagram.
    \smallskip\item Likewise, the open circles in these tiles form Ferrers diagrams of partitions (after an up-down flip); for instance the tile $X_{11}$ in Figure~\ref{fig:quinv_dig_mod} contains \[\begin{matrix} {\circ} & \\ {\circ} & {\circ} \end{matrix}\] which we interpret as $2+1$. Again, letting $\tilde{\Lambda}_{ij}$ denote the partition formed by the open circles in  $X_{ij}$ and  $\tilde{\Lambda}:=( \tilde{\Lambda}_{ij})_{1 \leq i \leq j <n}$, we have $\psi_{\inv}(F) = (\rsort(F),\tilde{\Lambda})$. As above, $\inv(F)$ is the total number of open circles in the lattice paths diagram.
    \smallskip\item This picture also serves to visually demonstrate the box complementation relation. Observe that the partitions formed by the open and closed circles in each tile together constitute a rectangular configuration of dots. In a given tile $X=X_{ij}$, this rectangle has dimension $r \times s$ where $r, s$ are the numbers of type~II and type~I paths in $X$ respectively. In view of Lemma~\ref{lem:nesepic}, these numbers coincide precisely with $\NE_{ij}(T)$ and $\SE_{ij}(T)$ where $T=\rsort(F)$. 
\end{enumerate}
We leave the easy verifications of our assertions in Steps (4)-(6) above to the reader.
Figure~\ref{fig:quinv_dig_mod}(a) shows the decluttered diagram obtained from Figure~\ref{fig:quinv_dig}, while Figure~\ref{fig:quinv_dig_mod}(b) shows the solid and open circles that track intersections and non-intersections. The quinv and inv overlays obtained in steps 4, 5 of the above algorithm are exhibited in Figure~\ref{fig:overlay}. Observe that the quinv overlays match those in Figure~\ref{fig:gt-ssyt}, in conformity with the example in section~\ref{sec:eg}.

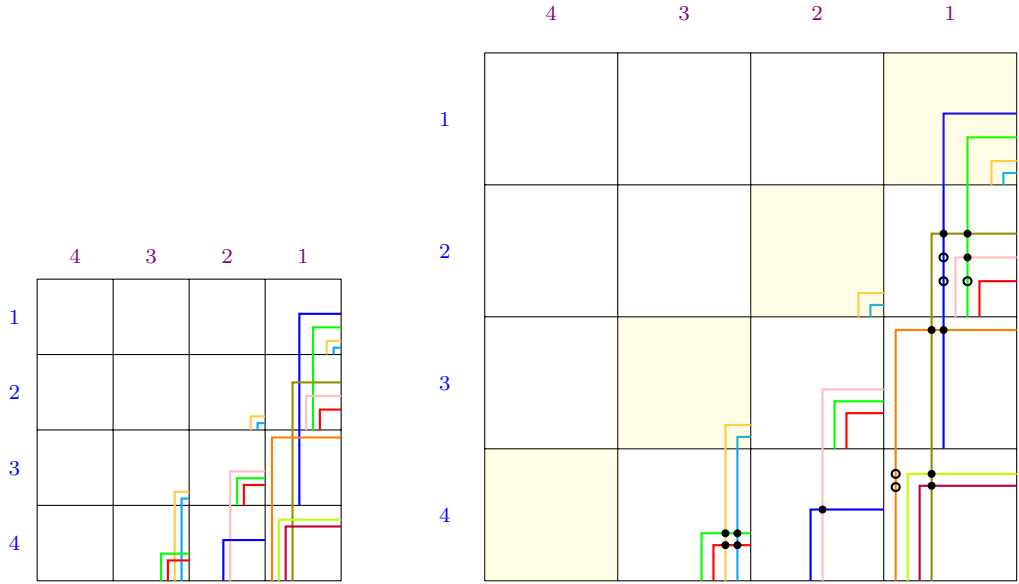
\begin{figure}[t]
\begin{center}
\begin{tikzpicture}[scale=1]
\draw(-1,0) grid (3,4);
\draw[color=cyan,thick](3,3.09)--(2.90,3.09)--(2.90,3);
\draw[color=cyan,thick] (2,2.09)--(1.90,2.09)--(1.90,2);
\draw[color=cyan,thick] (1,1.09)--(0.90,1.09)--(0.90,0.0);

%\draw[color=Mycolor8,thick](3,3.18)--(2.81,3.18)--(2.81,2.18)--(1.81,2.18)--(1.81,1.18)--(0.81,1.18)--(0.81,0.0);
\draw[color=Mycolor8,thick](3,3.18)--(2.81,3.18)--(2.81,3);
\draw[color=Mycolor8,thick] (2,2.18)--(1.81,2.18)--(1.81,2);
\draw[color=Mycolor8,thick] (1,1.18)--(0.81,1.18)--(0.81,0.0);

\draw[color=red,thick](3,2.27)--(2.72,2.27)--(2.72,2); 
\draw[color=red,thick](2,1.27)--(1.72,1.27)--(1.72,1); 
\draw[color=red,thick](1,0.27)--(0.72,0.27)--(0.72,0.0);
\draw[color=green,thick](3,3.36)--(2.63,3.36)--(2.63,2);
\draw[color=green,thick](2,1.36)--(1.63,1.36)--(1.63,1); 
\draw[color=green,thick](1,0.36)--(0.63,0.36)--(0.63,0.0);
\draw[color=pink,thick](3,2.45)--(2.54,2.45)--(2.54,2); 
\draw[color=pink,thick](2,1.45)--(1.54,1.45)--(1.54,0.00);
\draw[color=blue,thick](3,3.54)--(2.45,3.54)--(2.45,1); 
\draw[color=blue,thick](2,0.54)--(1.45,0.54)--(1.45,0.00);
\draw[color=olive,thick](3,2.63)--(2.36,2.63)--(2.36,0.00);
\draw[color=purple,thick](3,0.72)--(2.27,0.72)--(2.27,0.00);
\draw[color=lime,thick](3,0.81)--(2.18,0.81)--(2.18,0.00);
\draw[color=orange,thick](3,1.90)--(2.09,1.90)--(2.09,0.00);
\foreach\i in{0,...,3}{
\pgfmathtruncatemacro{\k}{4-\i};
\draw[color=blue](-1.3,\i+0.5)node{$\scriptstyle \k$};};
\foreach\j in{0,...,3}{
\pgfmathtruncatemacro{\l}{4-\j};
\draw[color=violet] (\j-0.5,4.3)node {$\scriptstyle \l$};};
\end{tikzpicture}\hspace{1cm}
% second figure
\begin{tikzpicture}[scale=1.75]
\draw(-1,0) grid (3,4);
\draw[color=cyan,thick](3,3.09)--(2.90,3.09)--(2.90,3);
\draw[color=cyan,thick] (2,2.09)--(1.90,2.09)--(1.90,2);
\draw[color=cyan,thick] (1,1.09)--(0.90,1.09)--(0.90,0.0);
%\draw[color=Mycolor8,thick](3,3.18)--(2.81,3.18)--(2.81,2.18)--(1.81,2.18)--(1.81,1.18)--(0.81,1.18)--(0.81,0.0);
\draw[color=Mycolor8,thick](3,3.18)--(2.81,3.18)--(2.81,3);
\draw[color=Mycolor8,thick] (2,2.18)--(1.81,2.18)--(1.81,2);
\draw[color=Mycolor8,thick] (1,1.18)--(0.81,1.18)--(0.81,0.0);
\draw[color=red,thick](3,2.27)--(2.72,2.27)--(2.72,2); 
\draw[color=red,thick](2,1.27)--(1.72,1.27)--(1.72,1); 
\draw[color=red,thick](1,0.27)--(0.72,0.27)--(0.72,0.0);
\draw[color=green,thick](3,3.36)--(2.63,3.36)--(2.63,2);
\draw[color=green,thick](2,1.36)--(1.63,1.36)--(1.63,1); 
\draw[color=green,thick](1,0.36)--(0.63,0.36)--(0.63,0.0);
\draw[color=pink,thick](3,2.45)--(2.54,2.45)--(2.54,2); 
\draw[color=pink,thick](2,1.45)--(1.54,1.45)--(1.54,0.00);
\draw[color=blue,thick](3,3.54)--(2.45,3.54)--(2.45,1); 
\draw[color=blue,thick](2,0.54)--(1.45,0.54)--(1.45,0.00);
\draw[color=olive,thick](3,2.63)--(2.36,2.63)--(2.36,0.00);
\draw[color=purple,thick](3,0.72)--(2.27,0.72)--(2.27,0.00);
\draw[color=lime,thick](3,0.81)--(2.18,0.81)--(2.18,0.00);
\draw[color=orange,thick](3,1.90)--(2.09,1.90)--(2.09,0.00);
\fill (0.90,0.36) circle[radius=0.3mm];
\fill (0.90,0.27) circle[radius=0.3mm];
\fill (0.81,0.36) circle[radius=0.3mm];
\fill (0.81,0.27) circle[radius=0.3mm];
\fill (1.54,0.54) circle[radius=0.3mm];
\fill (2.36,0.72) circle[radius=0.3mm];
\fill (2.36,0.81) circle[radius=0.3mm];
\fill (2.45,1.90) circle[radius=0.3mm];
\fill (2.36,1.90) circle[radius=0.3mm];
\fill (2.45,2.63) circle[radius=0.3mm];
\fill (2.63,2.63) circle[radius=0.3mm];
\fill (2.63,2.45) circle[radius=0.3mm];
%new ones
\draw[thick] (2.45,2.45) circle[radius=0.3mm];
\draw[thick] (2.63,2.27) circle[radius=0.3mm];
\draw[thick] (2.45,2.27) circle[radius=0.3mm];
\draw[thick] (2.09,0.81) circle[radius=0.3mm];
\draw[thick] (2.09,0.71) circle[radius=0.3mm];
\fill [yellow, opacity=0.1] (-1,0) rectangle (0,1);
\fill [yellow, opacity=0.1] (0,1) rectangle (1,2);
\fill [yellow, opacity=0.1] (1,2) rectangle (2,3);
\fill [yellow, opacity=0.1] (2,3) rectangle (3,4);
\foreach\i in{0,...,3}{
\pgfmathtruncatemacro{\k}{4-\i};
\draw[color=blue](-1.3,\i+0.5)node{$\scriptstyle \k$};};
\foreach\j in{0,...,3}{
\pgfmathtruncatemacro{\l}{4-\j};
\draw[color=violet] (\j-0.5,4.3)node {$\scriptstyle \l$};};
\end{tikzpicture}
\caption{(a) The decluttering of Figure~\ref{fig:quinv_dig}. (b) Marked intersections and non-intersections. Here $\quinv(F)=12$ and $\inv(F)=5$.}
\label{fig:quinv_dig_mod}
\end{center}
\end{figure}

%The overlays
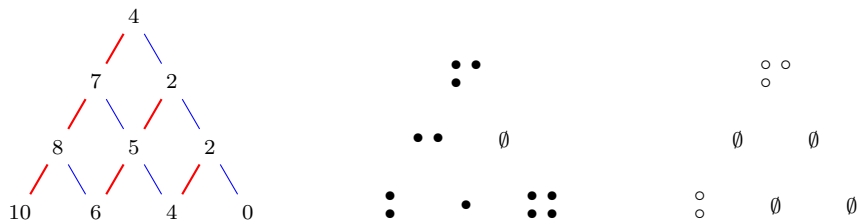
\begin{figure}[!h]
  \begin{center}
\begin{tikzpicture}[x={(1cm*0.5,-\rootthree cm*0.5)},y={(1cm*0.5,\rootthree cm*0.5)}]
  \draw(0,0)node(a00){$\scriptstyle{10}$};\draw(0,1)node(a01){$\scriptstyle{8}$};\draw(0,2)node(a02){$\scriptstyle{7}$};\draw(0,3)node(a03){$\scriptstyle{4}$};
  \draw(1,1)node(a11){$\scriptstyle{6}$};\draw(1,2)node(a12){$\scriptstyle{5}$};\draw(1,3)node(a13){$\scriptstyle{2}$};
  \draw(2,2)node(a22){$\scriptstyle{4}$};\draw(2,3)node(a23){$\scriptstyle{2}$};
  \draw(3,3)node(a33){$\scriptstyle{0}$};
  \foreach\i/\ii in{0/-1,1/0,2/1,3/2}
 \foreach\j/\jj in{0/-1,1/0,2/1,3/2}{
  \ifnum\i<\j \draw[color=red,thick](a\i\jj)--(a\i\j); \fi
  \ifnum\i>\j\else\ifnum\i>0 \draw[color=blue](a\ii\j)--(a\i\j);\fi\fi
}
\end{tikzpicture}
%% \qquad
%% \begin{tikzpicture}
%% \ytableausetup{mathmode, boxsize=0.3em}
%%   \draw(0,0) node {
%% \ytableaushort{
%%   1 1 1 1 2 2 2 3 4 4, 
%%   2 2 3 3 3 4,
%%   3 3 4 4}};
%% \end{tikzpicture}
\qquad\quad
\begin{tikzpicture}[x={(1cm*0.5,-\rootthree cm*0.5)},y={(1cm*0.5,\rootthree cm*0.5)}]
    \ytableausetup{mathmode, boxsize=0.5em}
  \draw(0,0)node{$\begin{smallmatrix} \bullet \\ \bullet \end{smallmatrix}$};\draw(0,1)node{$\begin{smallmatrix} \bullet & \bullet \end{smallmatrix}$};\draw(0,2)node{$\begin{smallmatrix} \bullet & \bullet \\ \bullet \end{smallmatrix}$};
  \draw(1,1)node{$\scriptstyle{\bullet}$};\draw(1,2)node{$\scriptstyle{\emptyset}$};
  \draw(2,2)node{$\begin{smallmatrix} \bullet & \bullet \\ \bullet & \bullet \end{smallmatrix}$};
\end{tikzpicture}
\qquad\quad
\begin{tikzpicture}[x={(1cm*0.5,-\rootthree cm*0.5)},y={(1cm*0.5,\rootthree cm*0.5)}]
  \draw(0,0)node{$\begin{smallmatrix} \circ \\ \circ \end{smallmatrix}$};\draw(0,1)node{$\scriptstyle{\emptyset}$};
  \draw(0,2)node{$\begin{smallmatrix} \circ & \circ \\ \circ \end{smallmatrix}$};
  \draw(1,1)node{$\scriptstyle{\emptyset}$};\draw(1,2)node{$\scriptstyle{\emptyset}$};
  \draw(2,2)node{$\scriptstyle{\emptyset}$};
\end{tikzpicture}
\caption{(a) $T=\rsort(F)$ for $F$ in Figure~\ref{fig:quinv_dig}. (b) The quinv overlays obtained from the solid circles in Figure~\ref{fig:quinv_dig_mod} (compare Figure~\ref{fig:gt-ssyt}). (c) The inv overlays obtained from the open circles in Figure~\ref{fig:quinv_dig_mod}.}
\label{fig:overlay}
\end{center}
\end{figure}

\printbibliography

\end{document}